\documentclass[10pt,a4paper]{amsart}

\allowdisplaybreaks[4]

\usepackage{amsmath,amsfonts,amsthm,amssymb,bbm,microtype,mlmodern,graphicx}

\usepackage{color}

\newtheorem{theorem}{Theorem}
\newtheorem{lemma}{Lemma}
\newtheorem{proposition}{Proposition}
\newtheorem*{claim}{Claim}

\DeclareMathOperator{\supp}{supp}

\newcommand{\Meng}[2]{\left\{#1\mathrel{}\middle|\mathrel{}#2\right\}}
\newcommand{\abs}[1]{\left\lvert#1\right\rvert}
\newcommand{\e}{\varepsilon}
\newcommand{\dd}{\delta}
\newcommand{\ee}{e}

\def\N{{\mathbbm N}}

\begin{document}

\title{Visiting time statistics}

\author{Maxim Kirsebom}
\address{Maxim Kirsebom, University of Bremen, Department 3 – Mathematics, Institute for Dynamical Systems, Bibliothekstr. 5, 28359 Bremen, Germany}
\email{kirsebom@uni-bremen.de}

\author{Philipp Kunde}
\address{Philipp Kunde, Oregon State University, Department of
  Mathematics, Kidder Hall 064, Corvallis, OR 97331, USA}
\email{kundep@oregonstate.edu}

\author{Tomas Persson}
\address{Tomas Persson, Centre for Mathematical Sciences, Lund University, Box 118, 221 00 Lund, Sweden}
\email{tomasp@gmx.com}

\thanks{We thank Jérôme Rousseau for asking a question, that led
  us to writing this paper.  This research was supported by the
  Lund--Hamburg funding scheme. We thank the University of
  Hamburg for hospitality.}


\begin{abstract}
  Many mixing dynamical systems $(X,T,\mu)$ are known to satisfy
  the hitting time statistics result
  \[
    \lim_{r \to 0} \mu \{\, x : \tau_{B(y,r)} (x) > t/\mu(B(y,r))
    \,\} = e^{-t},
  \]
  for $\mu$-almost every $y$, where $\tau_{B(y,r)} (x)$ is the
  first hitting time of $x$ to the ball $B(y,r)$. Taking a
  different point of view, we fix $x$ and consider
  $\tau_{B(y,r)} (x)$ as a function of $y$. We call this the
  visiting time of $y$ from $x$, i.e.\ the time it takes for $y$
  to get a visit from $x$ within a neighbourhood of radius
  $r$. We prove that
  \[
    \lim_{r \to 0} \mu \{\, y : \tau_{B(y,r)} (x) >
    t/\mu(B(y,r)) \,\} = e^{-t},
  \]
  for $\mu$-almost every $x$. As a byproduct we obtain a new
  method of proof for hitting time statistics.
\end{abstract}

\maketitle

\section{Introduction}

Consider an ergodic dynamical systems $(X,T,\mu)$, where, for
simplicity, we assume that $\mu$ is a probability measure. For a
measurable set $A \subset X$ of positive measure, we consider the first
hitting time to $A$, defined by
\[
  \tau_A (x) = \min \{\, n \geq 1 : T^n(x) \in A \,\}.
\]
In this setting, Kac's lemma \cite{Kac} states that
\[
  \int_A \tau_A \, \mathrm{d} \mu = 1.
\]
Hence, the expected value of $\tau_A$ on $A$ is $1/\mu(A)$.

To get more refined information about $\tau_A$, one may consider
the sets
\[
  \Bigl\{\, x \in A : \tau_A (x) \geq \frac{t}{\mu (A)}
  \,\Bigr\}
\]
and
\[
  \Bigl\{\, x \in X : \tau_A (x) \geq \frac{t}{\mu (A)}
  \,\Bigr\}.
\]
In the case that $A$ is a ball, it has been established in many
cases that
\begin{equation}
  \label{eq:returnstatistics}
  \lim_{r \to 0} \frac{1}{\mu (B(y,r))} \mu\Bigl\{\, x \in B(y,r) :
  \tau_{B(y,r)} (x) \geq \frac{t}{\mu (B(y,r))} \,\Bigr\} =
  e^{-t}
\end{equation}
and
\begin{equation}
  \label{eq:hittingstatistics}
  \lim_{r \to 0} \mu\Bigl\{\, x \in X : \tau_{B(y,r)} (x) \geq
  \frac{t}{\mu (B(y,r))} \,\Bigr\} = e^{-t},
\end{equation}
holds for $\mu$-almost every $y$. Results of the type
\eqref{eq:returnstatistics} are called (exponential) return time statistics,
and results of the form \eqref{eq:hittingstatistics} are called (exponential)
hitting time statistics.

A more refined version of \eqref{eq:hittingstatistics} is to
consider the limits
\begin{equation}
  \label{eq:poisson}
  \lim_{r \to 0} \mu \Bigl\{\, x \in X : p = \# \Bigl\{\, j \leq
  \frac{t}{\mu (B(y,r))} : T^j (x) \in B(y,r) \,\Bigr\} \,\Bigr\}
\end{equation}
where $p$ is a natural number. Note that for $p = 0$, we obtain
the limit in \eqref{eq:hittingstatistics}. When the limit in
\eqref{eq:poisson} is equal to $t^p e^{-t} / p!$, we say that the
hits to $B(y,r)$ are Poisson distributed (with intensity 1) in the
limit as $r \to 0$.

We now mention some results on return and hitting time
statistics. Collet \cite{Collet} obtained hitting time statistics
for various one dimensional systems. Pitskel' \cite{Pitskel}
proved that for Markov chains (in particular for subshifts of
finite type), hits are Poisson distributed for almost all $y$,
and in particular, he proved that \eqref{eq:hittingstatistics}
holds for such systems. Hirata \cite{Hirata} proved Poisson
distribution for hitting for Axiom A diffeomorphisms, and Hirata,
Saussol and Vaienti \cite{Hirataetal} extended the results on
Poisson distribution of hits to some systems that are not
exponentially mixing.

Bruin, Saussol, Troubetzkoy and Vaienti \cite{Bruinetal} showed
that return time statistics for a first return map is the same as
for the original system, and they used this to obtain return time
statistics for some smooth interval maps, including some
quadratic maps. Denker, Gordin and Sharova \cite{Denkeretal}
obtained a Poisson distribution law for toral automorphisms.
Freitas, Freitas and Todd \cite{FreitasFreitasTodd} explored the
connections between hitting time statistics and extreme value
laws. Chazottes and Collet \cite{ChazottesCollet} proved Poisson
distribution for hitting, for many non-uniformly hyperbolic
systems, including Young towers.

A survey of results on hitting and return times distributions has
been written by Haydn \cite{Haydn}.

Of a different, but related, type, is the recent result of Holland
and Todd \cite{HollandTodd}. For some interval maps, they
considered the sets
\[
  \Bigl\{\, x \in X : p = \# \Bigl\{\, j \leq \frac{t}{\mu
    (B(x,r))} : T^j (x) \in B(x,r) \,\Bigr\} \,\Bigr\},
\]
that is, the same set as in \eqref{eq:poisson}, but with $y =
x$. Their result is that in the limit as $r \to 0$, the measure
of these sets satisfies a Poisson-like distribution, more
precisely an average of Poisson distributions with different
intensities. As a special case, their result implies that the
limit
\begin{equation}
  \label{eq:recurrencestatistics}
  \lim_{r \to 0} \mu\Bigl\{\, x \in X : \tau_{B(x,r)} (x) \geq
  \frac{t}{\mu (B(x,r))} \,\Bigr\}
\end{equation}
is equal to an average of $e^{-\lambda t}$ over different
$\lambda$. By defining $r(x,\delta)$ by the requirement that
$\mu (B(x, r(x,\delta))) = \delta$ for all $x$, a small
modification of the proof of Holland and Todd implies that
\begin{equation}
  \label{eq:recurrencestatistics2}
  \lim_{\delta \to 0} \mu\Bigl\{\, x \in X : \tau_{B(x,r
    (x,\delta))} (x) \geq \frac{t}{\mu (B(x, r(x,\delta)))}
  \,\Bigr\} = e^{-t}.
\end{equation}
We will refer to results of this nature as \emph{recurrence time statistics}.

To put the limits \eqref{eq:hittingstatistics} and
\eqref{eq:recurrencestatistics2} in a unified context, consider
the set
\[
  E_\delta = \Bigl\{\, (x,y) \in X \times X :
  \tau_{B(y,r(y,\delta))} (x) \geq \frac{t}{\mu
    (B(y,r(y,\delta)))} \,\Bigr\},
\]
which is also illustrated in Figure~\ref{fig:sections}. The hitting time statistics results say that for almost all $y$,
the horizontal section of $E_\delta$ at $y$ has measure which
converges to $e^{-t}$ as $\delta \to 0$, whereas the variant
\eqref{eq:recurrencestatistics2} of Holland's and Todd's result
gives the limit of the measure of the diagonal of $E_{\delta}$.

In cases where hitting time statistics is known it therefore follows that
$\mu \times \mu (E_\delta) \to e^{-t}$ as $\delta \to 0$, and it
is therefore natural to expect that we get the same result as for
hitting time, if we replace horizontal sections by vertical
sections. This indeed turns out to be the case. In this paper
we prove for some exponentially mixing systems on the interval,
that for $\mu$ almost every $x$,
\[
  \lim_{\delta \to 0} \mu \Bigl\{\, y \in X :
  \tau_{B(y,r(y,\delta))} (x) \geq \frac{t}{\mu
    (B(y,r(y,\delta)))} \,\Bigr\} = e^{-t}.
\]
We call this \emph{visiting time statistics}. Hence, for the set
$E_\delta$ we therefore know that, in the limit as
$\delta \to 0$, the diagonal section and almost every horizontal
and almost every vertical section, has measure $e^{-t}$. This is
illustrated in Figure~\ref{fig:sections}.

\begin{figure}
  \includegraphics{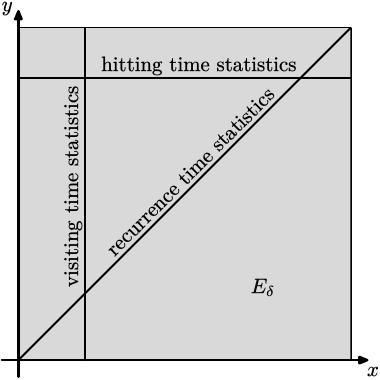}
  \caption{Illustration of how various results relate to the set
    $E_\delta$. Hitting time statistics are related to the
    horizontal sections of $E_\delta$, visiting time statistics
    are related to the vertical sections of $E_\delta$, and the
    result of Holland and Todd is related to the diagonal
    section.}
  \label{fig:sections}
\end{figure}

Our method is based on generating functions, and as a biproduct,
we show how this approach can be used to prove Poisson
distribution results for hitting times.

\section{Visiting time statistics}

We assume that $T \colon [0,1] \to [0,1]$ is piecewise expanding
with at most finitely many branches. More precisely, we assume that $T$ is $C^2$
everywhere except on a set of at most finitely many points and
$|T'|$ is bounded from above and bounded away from 1 on the same
set.  Assume also that $\mu$ is an ergodic invariant measure for
which there is a constant $C > 1$ such that
\begin{equation}
  \label{eq:Lebesgue}
  C^{-1} r < \mu (B(x,r)) < C r
\end{equation}
holds for all $x\in [0,1]$ and all $r < 1$. In this setting it is
known \cite{HofbauerKeller, Rychlik} that $\mu$ is exponentially
mixing for $BV$ functions against $L^1$, that is that for some
$\tau, \tilde{C}>0$,
\[
  \biggl| \int f \circ T^n g \, \mathrm{d} \mu - \int f \,
  \mathrm{d} \mu \int g \, \mathrm{d} \mu \biggr| \leq \tilde{C} e^{-\tau
    n} \lVert f \rVert_{L^1} \lVert g \rVert_{BV},
\]
holds for all $n$ and all functions $f$ and $g$ with
$\lVert f \rVert_{L^1}, \lVert g \rVert_{BV}<\infty$.

\begin{theorem}[Visiting time statistics]
  \label{the:vts_extension}
  With the assumptions on $T$ and $\mu$ as above, we have for
  $\mu$-almost every $x$ that
  \[
    \lim_{r \to 0} \mu \{\, y : \tau_{B(y,r)} (x) >
    t/\mu(B(y,r)) \,\} = e^{-t},
  \]
  for all $t \ge 0$.
\end{theorem}

We give the proof of Theorem~\ref{the:vts_extension} in
Section~\ref{sec:vtsproof_extension}. However, to prove
Theorem~\ref{the:vts_extension}, we first establish the result in
an alternative form, that we state below and prove in
Section~\ref{sec:vtsproof}.  Before stating this theorem, we
notice that from the assumptions on $\mu$, it follows that for
each natural number $n$ and each $t > 0$, there is a well-defined
function $r_n(\cdot, t)$ such that $\mu (B(y,r_n(y,t))) = t/n$.

\begin{theorem}[Visiting time statistics]
  \label{the:vts}
  With the assumtions on $T$ and $\mu$ as above, we have for
  $\mu$-almost every $x$ that
  \[
    \lim_{n \to \infty} \mu \{\, y : \tau_{B(y,r_n(y,t))} (x) >
    t/\mu(B(y,r_n(y,t))) \,\} = e^{-t},
  \]
  for all $t \ge 0$.
\end{theorem}

\subsection{Hitting time statistics}

Our method to prove Theorem~\ref{the:vts} makes it possible to
give a simple and direct proof of hitting time statistics. For
the results in this section, we can consider more general
systems, and we need only a somewhat weaker form of mixing. We
say that $\mu$ is exponentially mixing for intervals against
$L^1$ if there are numbers $\tau, \tilde{C} > 0$, such that
\[
  \biggl| \int f \circ T^n \mathbbm{1}_I \, \mathrm{d} \mu -
  \mu(I) \int f \, \mathrm{d} \mu \biggr| \leq \tilde{C} e^{-\tau
    n} \lVert f \rVert_{L^1},
\]
holds for all $n$, all functions $f$ with
$\lVert f \rVert_{L^1} < \infty$, and all intervals
$I \subset [0,1]$.

We shall prove the following theorem.

\begin{theorem}[Hitting time statistics]
  \label{the:hts}
  Let $([0,1],T,\mu)$ be exponentially mixing for intervals
  against $L^1$, and suppose that $y \in \supp \mu$ is a point
  such that there are constants $c_1, r_0 > 0$ for which
  \begin{equation}\label{eq:NoReturnCond}
  B(y,r) \cap T^k B(y,r) = \emptyset,
  \end{equation}
  for all $k \leq - c_1 \log r$ and all $r < r_0$.

  Then
  \[
    \lim_{r \to 0} \mu \{\, x : \tau_{B(y,r)} (x) > t / \mu
    (B(y,r)) \, \} = e^{-t}
  \]
  for all $t \geq 0$.
\end{theorem}
We mention two examples for which this theorem applies. 

If $T$ is piecewise expanding and $C^2$, with $|T'|$ bounded from
above and away from 1, and if $\mu$ is an ergodic Gibbs measure,
then the system is exponentially mixing for bounded variation
agains $L^1$. Moreover, assumption \eqref{eq:NoReturnCond} is
satisfied for $\mu$ a.e.\ point $y$, which follows from Holland,
Kunde, Kirsebom and Persson \cite[Claim in proof of
Proposition~13.1]{Hollandetal}. Hence, Theorem~\ref{the:hts}
applies in this case for $\mu$ a.e.\ $y$.

Let $T_a = ax(1-x)$ where $a$ is a Benedicks--Carlesson parameter
and $\mu$ is the invariant measure which is absolutely continuous
with respect to Lebesgue measure.  Then exponential mixing holds
for intervals against $L^1$, see
Section~\ref{sec:appendix}. Suppose $y = \frac{1}{2}$, then there are positive constants $c$ and
$\gamma$ such that
$|y - T^k_a (y)| > c e^{-\gamma k}$ for all $k \in \mathbb{N}$ \cite[Section~6]{BenedicksCarleson}. Since
$T^k_a B(y,r) \subset B(T^k_a (y), 4^k r)$, we have
\eqref{eq:NoReturnCond} as long as
$r (1 + 4^k) < c e^{-\gamma k}$. That is certainly satisfied as long as
$k \leq - c_1 \log r$, for some $c_1 > 0$. In conclusion,
Theorem~\ref{the:hts} applies for $y = \frac{1}{2}$. Note that exponential HTS for $T_a$ at $y=\frac{1}{2}$ was proven already in \cite[Corollary 3]{FreitasFreitasTodd} by translating an extreme value law for $T_a$ into a HTS result. Our point is that with our method the same result can be obtained in a direct way.

With a little bit more effort, we can prove the following
theorem.
\begin{theorem}
  \label{the:poisson}
  With the same assumptions as in Theorem~\ref{the:hts}, we get a
  Poisson distribution for the number of hits. For any integer $p
  \geq 0$,
  \[
    \lim_{r \to 0} \mu \{\, x : p = \# \{\, j \leq t/\mu(B(y,r))
    : T^j (x) \in B(y,r) \,\} \,\} = \frac{t^p e^{-t}}{p!} .
  \]
\end{theorem}

Results of this type are not new, although we are not aware of a
result which immediately implies the above theorem. For instance,
Freitas, Freitas and Magalhães \cite[Corollary
4.4]{FreitasFreitasMagalhaes} give the same result under similar
assumptions. Our main point with stating and proving
Theorems~\ref{the:hts} and \ref{the:poisson} is that the method
of proof is new.

If $(M,T,\mu)$ is a piecewise expanding higher dimensional system
satisfying the assumptions of Saussol \cite[Theorem~5.1]{Saussol}
or Thomine \cite[Theorem~2.4]{Thomine}, then $\mu$ is
exponentially mixing for $BV$ against $L^1$, and
Theorems~\ref{the:hts} and \ref{the:poisson} offer a new strategy
of proof for hitting time and Poisson statistics in cases where
assumption \eqref{eq:NoReturnCond} can be verified. As far as we
know, hitting time statistics and Poisson distributions are not
known for this class of systems.

\section{Proof of Theorem~\ref{the:vts}}
\label{sec:vtsproof}

Throughout the proof, $c$ will denote a generic constant whose value may vary from line to line or even appear multiple times in the same line while representing different numerical values.

We fix $t > 0$ and $n$. Note that $t$ will on occasions be integrated into the generic constant $c$ when the $t$-dependence is not important. Denote $B_y := B(y,r_n (y,t))$,
\begin{equation}\label{eq:A_n}
  A_n := \biggl\{\, (x,y) : \tau_{B_y} (x) > \frac{t}{\mu (B_y)}
  \,\biggr\}
\end{equation}
and
\[
  A_{n,x} := \biggl\{\, y : \tau_{B_y} (x) > \frac{t}{\mu (B_y)}
  \,\biggr\}.
\]
Note that $A_n=\{ (x,y) : \tau_{B_y} (x) > n \}$ and 
\begin{equation*}
	\mu(A_{n,x})=\int \mathbbm{1}_{A_n}(x,y)\, \mathrm{d} \mu (y)=: f_n(x).
\end{equation*}
Note also that in this notation, Theorem~\ref{the:vts} states that for $\mu$-a.e.\ $x\in [0,1]$
\begin{equation*}
	\lim_{n\to\infty} \mu(A_{n,x})=\lim_{n\to\infty} f_n(x)=e^{-t},
\end{equation*}
for all $t>0$. We recognize that
\[
  \mathbbm{1}_{A_n} (x,y) = \prod_{k=1}^n (1 - \mathbbm{1}_{B_y}
  (T^k (x))),
\]
and by expanding the product, we get
\begin{align}\label{eq:ExpProduct}
  \mathbbm{1}_{A_n} (x,y)
  &= 1 - \sum_{1 \leq k_1 \leq n}
    \mathbbm{1}_{B_y} (T^{k_1} (x)) + \sum_{1 \leq k_1 < k_2 \leq n}
    \mathbbm{1}_{B_y} (T^{k_1} (x)) \mathbbm{1}_{B_y} (T^{k_2}
    (x)) \nonumber\\ & \hspace{4cm} + \ldots + (-1)^n \mathbbm{1}_{B_y}
              (T^{1} (x)) \ldots \mathbbm{1}_{B_y} (T^n (x))\nonumber \\
  &= 1 + \sum_{m = 1}^n (-1)^m S_{m,n} (x,y),
\end{align}
where
\[
  S_{m,n} (x,y) = \sum_{1 \leq k_1 < \ldots < k_m \leq n}
  \mathbbm{1}_{B_y} (T^{k_1} (x)) \ldots \mathbbm{1}_{B_y}
  (T^{k_m} (x)).
\]
Hence the quantity that we need to estimate the limit of is given by
\[
  f_n (x) = 1 + \sum_{m=1}^{n} (-1)^m \int S_{m,n} (x,y) \,
  \mathrm{d} \mu (y).
\]

For a sequence $1 \leq k_1 < k_2 < \ldots < k_m \leq n$, we write
$k = (k_1, \ldots, k_m)$, and we let $K_{m,n}$ denote the set of all
such sequences. We consider the set
\[
  \tilde{K}_{m,n} := \{\, k \in K_{m,n}: k_{j+1} - k_j > c_1 \log
  n \, \}.
\]

\begin{lemma}
  \label{lem:shortreturn}
  If we choose $c_1$ small enough, then there exists a sequence of sets $E_n$ such
  that $\mu(E_n) \to 1$ as $n \to \infty$, and if $y \in E_n$ and $k \in K_{m,n} \setminus \tilde{K}_{m,n}$ for any $m\leq n$, then,
  \[
    \mathbbm{1}_{B_y} (T^{k_1} (x)) \ldots \mathbbm{1}_{B_y}
    (T^{k_m} (x)) = 0
  \]
  for all $x\in [0,1]$.
\end{lemma}

\begin{proof}
In \cite{Hollandetal} the following claim is proven under the same assumptions as in Theorem~\ref{the:vts}.

\begin{claim}[\cite{Hollandetal}, p. 3968] There exists a number $\kappa>0$ such that for $\mu$-almost every $y$, there exists an $n_0(y)\in \mathbbm{N}$, such that
	  \begin{equation}
		\label{eq:noshortreturns}
		B_y \cap T^l B_y = \emptyset, \quad \text{for all } l \leq
		\kappa \log n,
	\end{equation}
	for all $n\geq n_0(y)$.
	\end{claim}
	Let $\kappa>0$ be given by the claim and let $E_n$ be the set of $y$ such that
	\eqref{eq:noshortreturns} holds. Then $\mu (E_n) \to 1$ as $n \to \infty$, by the claim. Suppose $k\in K_{m,n}\backslash \tilde{K}_{m,n}$ and choose $c_1\leq \kappa$. Then there exists a $j<m$ s.t. $k_{j+1}-k_j\leq \kappa\log n$. Suppose $y\in E_n$, then we know that 
	\[
	B_y\cap T^{k_{j+1}-k_j}B_y=\emptyset.
	\]
	Applying $T^{-k_{j+1}}$ and rewriting the equation as characteristic functions gives
	\[
	\mathbbm{1}_{B_y}(T^{k_j}x)\mathbbm{1}_{B_y}(T^{k_{j+1}}x)=0, \quad \text{for all }x,
	\]
	which clearly implies 
	\[
	\mathbbm{1}_{B_y}(T^{k_1}x)\dots \mathbbm{1}_{B_y}(T^{k_m}x)=0, \quad \text{for all }x.
	\]
%
%
%
%
%
%
\end{proof}
Set $\tilde{f}_n (x) := \int_{E_n} \mathbbm{1}_{A_n} (x,y) \,\mathrm{d} \mu(y)$ such that 
\[
	f_n (x) = \tilde{f}_n (x) +
	\int_{\complement E_n} \mathbbm{1}_{A_n} (x,y) \, \mathrm{d} \mu(y).
\]
It follows that
\[
|f_n-\tilde{f}_n| \leq 1 - \mu(E_n).
\]
and since $\mu(E_n)\to 1$ by Lemma~\ref{lem:shortreturn}, Theorem~\ref{the:vts} follows if $\lim_{n\to\infty}\tilde{f}_n(x)=e^{-t}$ for $\mu$-almost every $x$. The advantage of working with $\tilde{f}_n$ is that with 
\[
	\tilde{S}_{m,n} (x,y) := \sum_{k \in \tilde{K}_{m,n}}
	\mathbbm{1}_{B_y} (T^{k_1} (x)) \ldots \mathbbm{1}_{B_y}
	(T^{k_m} (x)),
\]
it follows from \eqref{eq:ExpProduct} and Lemma~\ref{lem:shortreturn} that
\[
	\tilde{f}_n(x)=\mu(E_n)+\sum_{m=1}^{n} (-1)^m \int_{E_n}
	\tilde{S}_{m,n} (x,y) \, \mathrm{d} \mu (y).
\] 
since if $y \in E_n$ then $S_{m,n} (x,y) = \tilde{S}_{m,n} (x,y)$. Not only does $\tilde{S}_{m,n} (x,y)$ contain less summands than $S_{m,n}$, but due to the gaps ensured by $\tilde{K}_{m,n}$ the products are amenable to applications of exponential mixing.

In order to estimate the limit of $\tilde{f}_n$ we will be working with two functions that bound $\tilde{f}_n$ from above and below and for which we will need to include even less summands. This is done as follows. Fix $\alpha \in (0,1)$ and let $[\cdot]$ denote the floor function. Let $M_n^\mathrm{odd}$ and
$M_n^\mathrm{even}$ be the two integers in the set
\[
  \{ [(\log n)^\alpha], [(\log n)^\alpha] + 1\}
\]
chosen such that
$M_n^\mathrm{odd}$ is odd and $M_n^\mathrm{even}$ is even.
We define $g_n^\mathrm{odd}$ and $g_n^\mathrm{even}$ by
\[
  g_{n}^\mathrm{odd} (x) = \mu(E_n) + \sum_{m=1}^{M_n^\mathrm{odd}} (-1)^m \int_{E_n} \tilde{S}_{m,n}
  (x,y) \, \mathrm{d} \mu (y),
\]
and similarly for $g_n^\mathrm{even}$. Since $n>M_n^\mathrm{odd},M_n^\mathrm{even}$, we have by the inclusion--exclusion principle, that
\[
  1 + \sum_{m=1}^{M_n^\mathrm{odd}} (-1)^m \tilde{S}_{m,n} (x,y)
  \leq
  1 + \sum_{m=1}^{n} (-1)^m \tilde{S}_{m,n} (x,y)
  \leq
  1 + \sum_{m=1}^{M_n^\mathrm{even}} (-1)^m \tilde{S}_{m,n} (x,y) .
\]
Integrating over $E_n$ we get that
\[
g_n^\mathrm{odd} \leq \tilde{f}_n \leq g_n^\mathrm{even},
\]
and it follows that if
$g_n^\mathrm{odd} (x), g_n^\mathrm{even} (x) \to e^{-t}$ then so does $\tilde{f}_n(x)$. The argument for this convergence is the same for
$g_n^\mathrm{odd}$ and $g_n^\mathrm{even}$ and we shall therefore
only write $M_n$ and $g_n$ from now on.

%
Next we will decompose $g_n$ in a useful way. Via the binomial formula and since $\mu(E_n)\to 1$ we see that , 
\[
\mu(E_n)+\mu(E_n)\sum_{m=1}^n\binom{n}{m}\frac{(-t)^m}{n^m}=\mu(E_n)\biggl(1-\frac{t}{n}\biggr)^n\to e^{-t}
\]
which motivates the definitions
\begin{align*}
  \Xi_n &:= \mu (E_n) \sum_{m=M_n+1}^{n} \binom{n}{m}
  \frac{(-t)^m}{n^{m}},\\
 \Psi_n &:= \mu (E_n)
\sum_{m=1}^{M_n} \biggl( \binom{n}{m} - \# \tilde{K}_{m,n}
\biggr) \frac{(-t)^m}{n^{m}}
\end{align*}
as well as
\begin{align*}
\Delta_{x,m} &:= \int_{E_n} \tilde{S}_{m,n}(x,y) \, \mathrm{d} \mu (y)  -
\#\tilde{K}_{m,n}\mu (E_n) \frac{t^m}{n^m}.\\
\lambda_n(x)&:=\sum_{m = 1}^{M_n}(-1)^m \Delta_{x,m}.
\end{align*}
Inserting these definitions it is easily verified that
\begin{align*}
  g_n (x)
  &= \mu (E_n) \biggl(1 - \frac{t}{n}
    \biggr)^{n} - \Xi_n - \Psi_n + \lambda_n(x).
\end{align*}
The following two results simplify the asymptotics of $g_n$.
\begin{lemma}
	\label{lem:binomial}
	For $m\leq M_n$  we have
	\[
	0 \leq 1 - \frac{\# \tilde{K}_{m,n}}{\binom{n}{m}} \leq c
	\frac{(\log n)^{1 + 2 \alpha}}{n}.
	\]
	for some $c>0$.
\end{lemma}

\begin{proof}
	Note that $\#K_{m,n}=\binom{n}{m}$ and hence
	\[
	1-\frac{\#\tilde{K}_{m,n}}{\binom{n}{m}}=\frac{\#\complement\tilde{K}_{m,n}}{\binom{n}{m}}.
	\]
	For a $k\in\complement\tilde{K}_{m,n}$, $m-1$ of its elements have no restrictions on their location while one $j\leq m-1$ must satisfy $k_{j+1}-k_{j}\leq c_1\log n$. We distribute the $m-1$ elements freely among $n$ locations which gives $\binom{n}{m-1}$ combinations. The last element can then at most be placed within a $c_1\log n$ neighborhood of each of the $m-1$ other elements. Using that $m\leq M_n\leq \tilde{c}(\log n)^{\alpha}$ for some $\tilde{c}>1$, this gives
	\begin{align*}
	\frac{\# \complement \tilde{K}_{m,n}}{\binom{n}{m}} \leq
	\frac{\binom{n}{m - 1}(m-1)(2 c_1 \log n)}{ \binom{n}{m}} &= (m-1) 2 c_1 \log n \frac{m}{n-(m-1)}\\
	&\leq c \frac{(\log n)^{1 + 2
			\alpha}}{n}, 
	\end{align*}
	since for some constant $\hat{c}\in (0,1)$ we have $n-\tilde{c}(\log n)^{\alpha}>\hat{c} n$.
\end{proof}

\begin{lemma}\label{lem:DeviationZero}
	  We have that $\Xi_n, \Psi_n \to 0$ as $n \to \infty$.
\end{lemma}
\begin{proof}
	For $\Xi_n \to 0$, expand $(1-\frac{t}{n})^n$ via the binomial formula as well as through the $M_n$'th order Taylor polynomial to see that $\Xi_n$ corresponds to the remainder term in Taylor's Theorem. With for instance a remainder term of Lagranges's type, a simple computation shows that the remainder vanishes.
	
	For $\Psi_n \to 0$, simply insert the estimate from Lemma~\ref{lem:binomial} in the definition of $\Psi_n$ to immediately get the convergence.
\end{proof}

By the decomposition of $g_n$ it now follows that
$g_n(x)\to e^{-t}$ if $\lambda_n(x)\to 0$. We will obtain this
convergence $\mu$-almost surely, along geometric sequences, by
considering the second moment of $\lambda_n$, i.e.

%

\begin{equation}
  \label{eq:Lambda}
  \Lambda_n := \int \lambda_n^2(x)\, \mathrm{d} \mu (x).
\end{equation}

This argument is part of the next proposition which also completes the proof of Theorem~\ref{the:vts} pending the confirmation of one estimate.

\begin{proposition}\label{prop:crit}
  If for some $s>0$ we have that for all $n$ sufficiently large the upper bound
  \begin{equation}
  	\label{eq:Lambda_n-Bound}
  	\Lambda_n \leq \frac{c}{n^s},
  \end{equation}
  holds, then for $\mu$-almost every $x\in X$ we have that $f_n(x)\to e^{-t}$ for all $t>0$.
\end{proposition}

\begin{proof}
Suppose $\Lambda_n\leq\frac{c}{n^s}$ for some $c>0,s>0$. Then 
\[
	\sum_{k=1}^{\infty} \Lambda_{l_k} =\sum_{k=1}^{\infty} \int \lambda_{l_k}^2 \, \mathrm{d} \mu<\infty,
\]
for all integer sequences of the form $l_k:=l_k(a):=\lceil a^k \rceil$ with $a>1$ where $\lceil\cdot\rceil$ denotes the roof function. The Markov inequality then gives that for any $\delta>0$
\[
	\mu\big \{x:\lambda_{l_k}^2(x)>\delta \big \}\leq \frac{1}{\delta} \int \lambda_{l_k}^2\,\mathrm{d}\mu.
\] 
This implies $\sum_{k=1}^{\infty}\mu\big \{x:\lambda_{l_k}^2(x)>\delta \big \}<\infty$ which means that almost surely $\lambda_{l_k}^2<\delta$ for all $k$ sufficiently large by the Borel--Cantelli Lemma. Since $\delta>0$ is arbitrary we get that almost surely $\lambda_{l_k}\to 0$ as $k\to\infty$ and hence also that almost surely $g_{l_k}\to e^{-t}$ as $k\to \infty$. By previous discussion this means that almost surely $\tilde{f}_{l_k}\to e^{-t}$ and in turn $f_{l_k}\to e^{-t}$. This almost sure convergence holds for a fixed $t>0$ and a given choice of $a$, hence the set of measure one for which it holds depends on $t$ and $a$. We denote these sets by $V_{t,a}$. For the rest of this proof we also write $f_n(x,t)$ for $f_n$ and $B_{y,n,t}$ for $B_y$ to emphasize dependecies. Since
\[
\tilde{V}:=\bigcap_{t\in\mathbbm{Q}_{>0}}\bigcap_{a\in\mathbbm{Q}_{>1}} V_{t,a}
\]
is a countable intersection we immediately get the convergence almost surely for all rational $t>0$ and all rational $a>1$.

%

Fix $x \in \tilde{V}$ and let $a\in\mathbbm{Q}_{>1}$. For any $n$, there is a $k$ such that
$l_{k}(a) < n \leq l_{k+1}(a)$. If $n$ (and so $k$) is large enough, we
have that
\[
  \frac{l_{k+1}(a)}{a^2} \leq n \leq a^2 l_k (a),
\]
which implies that
\[
	\mu(B_{y,l_{k+1},a^2t})\geq \mu(B_{y,n,t})\geq \mu(B_{y,l_k,t/a^2}) 
\]
Recall that by definition,
\[
	f_n(x,t)=\mu\left\{y:\tau_{B_{y,n,t}}(x)>t/\mu(B_{y,n,t})\right\}.
\]
When comparing this to 
\[
	f_{l_{k+1}}(x,a^2 t)=\mu\left\{ y: \tau_{B_{y,l_{k+1},a^2t}}(x)>a^2 t/\mu\left(B_{y,l_{k+1},a^2 t}\right)\right\},
\]
we are considering the first return to a larger set (i.e.\ $\tau_{B_{y,n,t}}(x) \geq \tau_{B_{y,l_{k+1},a^2t}}(x)$) on the left hand side of the inequality while the right hand side equals $l_{k+1}$, (i.e.\ $t/\mu(B_{y,n,t}) =n \leq l_{k+1} = a^2 t/\mu\left(B_{y,l_{k+1},a^2 t}\right)$). This implies
\[
  f_{l_{k+1}} (x,  a^2t) \leq f_{n} (x,t) \leq f_{l_k} (x, t/a^2),
\]
where the second inequality follows by an identical argument.
Hence $e^{-a^2 t} \leq \liminf f_{n} (x,t) \leq \limsup f_{n}
(x,t) \leq e^{-t/a^2}$. Since $a\in\mathbbm{Q}_{>1}$ is arbitrary, we conclude that
$\lim f_{n} (x,t) = e^{-t}$ for all $x\in \tilde{V}$ and all $t\in \mathbbm{Q}_{>0}$. Finally, since $f_{n} (x,t)$ converges almost surely to $e^{-t}$ for all positive rationals, and $t \mapsto f_{n} (x,t)$ is decreasing, it follows that $f_{n} (x,t)$ converges to $e^{-t}$ almost surely for all $t >0$.
%
%
%
%
\end{proof}

We now proceed to prove the bound in \eqref{eq:Lambda_n-Bound} which will conclude the proof of Theorem~\ref{the:vts} and which contains the majority of technical work. 

We let
\begin{align*}
  J_k&=J_k (x)
   := \int_{E_n}  \mathbbm{1}_{B_y} (T^{k_1} (x)) \ldots
    \mathbbm{1}_{B_y} (T^{k_m} (x)) \, \mathrm{d} \mu (y).
\end{align*}
Using \eqref{eq:Lambda}, it is easily verified that
\begin{align*}
	\Lambda_n
	& = \sum_{m_1,m_2=1}^{M_n} (-1)^{m_1+m_2} \int
	\Delta_{x,m_1} \Delta_{x,m_2} \, \mathrm{d} \mu (x) \\
	& = \sum_{m_1,m_2=1}^{M_n} (-1)^{m_1+m_2} \sum_{\substack{k \in \tilde{K}_{m_1,n} \\j \in
			\tilde{K}_{m_2,n} }} \biggl(
	\mu(E_n)^2 \biggl( \frac{t}{n} \biggr)^{m_1+m_2} - \mu(E_n)
	\biggl( \frac{t}{n} \biggr)^{m_2} \int J_k \,
	\mathrm{d} \mu \\
	& \hspace{2cm} - \mu(E_n)
	\biggl( \frac{t}{n} \biggr)^{m_1} \int J_j \,
	\mathrm{d} \mu + \int J_k J_j \,
	\mathrm{d} \mu \biggr).\\
\end{align*}
We will eventually estimate
$\Lambda_n\leq \lvert \Lambda_n\rvert$ and thereby get rid of the
alternating factor $(-1)^{m_1+m_2}$. We then get
\begin{align}
	\label{eq:numlambda}
	&\Lambda_n	 \leq  \sum_{m_1,m_2=1}^{M_n} \bigg\vert\#\tilde{K}_{m_1,n} \#\tilde{K}_{m_2,n}
	\mu(E_n)^2 \biggl( \frac{t}{n} \biggr)^{m_1+m_2}+ \sum_{\substack{k \in \tilde{K}_{m_1,n} \\j \in \tilde{K}_{m_2,n} }}\int J_k J_j \,
	\mathrm{d} \mu\nonumber\\
	& \hspace{3cm} - \mu(E_n) \#\tilde{K}_{m_2,n}\biggl( \frac{t}{n} \biggr)^{m_2}  \sum_{k\in \tilde{K}_{m_1,n}} \int J_k \,
	\mathrm{d} \mu \\
	& \hspace{3cm}- \mu(E_n) \#\tilde{K}_{m_1,n}\biggl(\frac{t}{n} \biggr)^{m_1} \sum_{j\in \tilde{K}_{m_2,n}} \int J_j \,
	\mathrm{d} \mu \bigg \vert.\nonumber
\end{align}
We see that we need upper and lower bounds on sums of the type
\[
	\sum_{k\in \tilde{K}_{m,n}} \int J_k \,
	\mathrm{d} \mu\quad\text{ and } \quad \sum_{\substack{k \in \tilde{K}_{m_1,n} \\j \in \tilde{K}_{m_2,n} }}\int J_k J_j \,
	\mathrm{d} \mu. 
\]
We begin with the first type which is the easier case. 
If $k$ is such that $k_{j+1} - k_j > c_2 \log n$ for some $c_2>0$ for all $j$,
then by repeated application of exponential mixing in which the $BV$-norm is applied to $\mathbbm{1}_{B_y}$ and the $L^1$-norm is applied to the remaining product, we obtain
\begin{equation}
	\label{eq:Ikbounds}
  \mu (E_n) \biggl(\frac{t}{n} - \frac{c}{n^{\tau c_2}} \biggr)^m
  \leq \int J_k \,
  \mathrm{d} \mu \leq \mu (E_n) \biggl(\frac{t}{n} + \frac{c}{n^{\tau
      c_2}} \biggr)^m.
\end{equation}
Hence, if $c_2 > \tau^{-1}$, then
\[
  \int J_k \,
  \mathrm{d} \mu \sim \biggl(\frac{t}{n}\biggr)^m \mu (E_n).
\]
For later estimates, it will actually be helpful to require
$c_2 > 4\tau^{-1}$. Also, the constant $c_1>0$ determined by
Lemma~\ref{lem:shortreturn} may be chosen arbitrarily small and
it will be helpful to require $\tau^{-1}>c_1$. So we fix
$c_2 > 4\tau^{-1} > \tau^{-1}> c_1$. For $1 \leq p \leq m-1$, let
$\hat{K}_{m,n,p}$ be those $k \in \tilde{K}_{m,n}$ such that
$k_{j+1} - k_j > c_2 \log n$ for all $j$ except for exactly $p$
times, when we have $c_1 \log n < k_{j+1} - k_j \leq c_2 \log n$.

\begin{proposition}
  \label{prop:Ik-sum}
  For $m\leq M_n$ we have for all sufficiently large $n$ that
  \begin{equation}
    \label{eq:Ik-sum}
    \biggl( 1 - \frac{c}{n^{1 - \varepsilon}} \biggr)
    \leq
    \frac{\sum_{k \in \tilde{K}_{m,n}} \int J_k \,
    	\mathrm{d} \mu}{\# \tilde{K}_{m,n}
      \mu(E_n) \bigl( \frac{c}{n} \bigr)^m}
    \leq  \biggl( 1 + \frac{c}{n^{\tau c_1
        - \varepsilon}} \biggr).
  \end{equation}
\end{proposition}

To prove Proposition~\ref{prop:Ik-sum}, we will need the
following lemma.

\begin{lemma}
  \label{lem:firstcounting}
  For $m\leq M_n$ we have
  \[
    \frac{\# \hat{K}_{m,n,p}}{\# \tilde{K}_{m,n}}  \leq c \frac{(\log
      n)^{p (1 + 2\alpha)}}{n^p}.
  \]
\end{lemma}

\begin{proof}
  We have
  \begin{align*}
  		\frac{\# \hat{K}_{m,n,p}}{\#\tilde{K}_{m,n}}&=\frac{\# \hat{K}_{m,n,p}}{\binom{n}{m}}\frac{\binom{n}{m}}{\#\tilde{K}_{m,n}}\\
  		&\leq\frac{\# \hat{K}_{m,n,p}}{\binom{n}{m}}\frac{1}{1-c\frac{(\log n)^{1+2\alpha}}{n}}\\
  		&\leq \frac{\# \hat{K}_{m,n,p}}{\binom{n}{m}}
  	\end{align*}
  by Lemma~\ref{lem:binomial}. To obtain an upper bound on $\# \hat{K}_{m,n,p}$ we argue as follows. Sequences in $\# \hat{K}_{m,n,p}$ have $m-p$ elements that give rise to $m-p-1$ gaps larger than $c_2\log n$. The trivial upper bound on the number of arrangements of the $m-p$ elements is $\binom{n}{m-p}$, i.e.\ simply ignoring the lower bound on the gaps. The remaining $p$ gaps are between $c_1\log n$ and $c_2\log n$. Again, by ignoring the lower bound of $c_1 \log n$ we ensure that we overcount. We imagine that $m-p$ elements have been distributed betwen $1$ and $n$. The first of the remaining $p$ elements can then at most be placed in a $c_2\log n$ neighborhood of each of the already distributed $m-p$ elements. For the following element there will be $m-p+1$ such neighborhoods, and for each further elements there will be an additional neighborhood in which it can be placed. However, to simplify the upper bound we will estimate that each of the $p$ elements can be placed in $m$ such neighborhoods. This gives the upper bound
  \begin{align*}
  		\frac{\# \hat{K}_{m,n,p}}{\binom{n}{m}}&\leq \frac{\binom{n}{m-p}(mc_2\log n)^p}{\binom{n}{m}}\\
  		&=(mc_2\log n)^p\frac{m(m-1)\dots (m-(p-1))}{(n-(m-p))\dots (n-(m-1))}\\
  		&\leq (mc_2\log n)^p\frac{m^p}{(n-m)^p}\\
  		&\leq ((\log n)^{\alpha} c_2\log n)^p\frac{(\log n)^{\alpha p}}{(n-(\log n)^{\alpha})^p}\\
  		&\leq c\frac{(\log n)^{p(1+2\alpha)}}{n^p}
  \end{align*}  
  since $n-(\log n)^{\alpha}\geq \tilde{c}n$ for some $\tilde{c}\in (0,1)$.
%
\end{proof}
Note that if $k \in \hat{K}_{m,n,p}$ then
\[
0 \leq \int J_k \,
\mathrm{d} \mu \leq \mu(E_n)
\biggl( \frac{t}{n} + \frac{c}{n^{\tau c_2}} \biggr)^{m-p}
\biggl(\frac{t}{n} +  \frac{c}{n^{\tau c_1}} \biggr)^p.
\]
The lower bound on $\int J_k \,
\mathrm{d} \mu$ is trivial while the upper bound follows from repeated application of exponential mixing in the same way as before.
\begin{proof}[Proof of Proposition~\ref{prop:Ik-sum}]
  We start with the upper bound on the sum of
  $\int J_k \, \mathrm{d} \mu$. Note that
  $\tilde{K}_{m,n}=\cup_{p=0}^{m-1} \hat{K}_{m,n,p}$ with the
  union being disjoint. We have by Lemma~\ref{lem:firstcounting}
  that
  \begin{align*}
    \sum_{k \in \tilde{K}_{m,n}} \int & J_k \,
    \mathrm{d} \mu
      \leq \sum_{p=0}^{m-1} \sum_{k \in \hat{K}_{m,n,p}} \mu (E_n) \biggl(
      \frac{t}{n} + \frac{c}{n^{\tau c_2}} \biggr)^{m-p}  \biggl(\frac{t}{n} + \frac{c}{n^{\tau c_1}} \biggr)^p \\
      &=\# \hat{K}_{m,n,0} \mu (E_n) \biggl(\frac{t}{n} + \frac{c}{n^{\tau c_2}} \biggr)^{m}\\
      & \phantom{.} \qquad + \sum_{p=1}^{m-1} \#\hat{K}_{m,n,p}\mu(E_n)\biggl(
      \frac{t}{n} + \frac{c}{n^{\tau c_2}} \biggr)^{m-p}\biggl(\frac{t}{n} + \frac{c}{n^{\tau c_1}} \biggr)^p\\
      &\leq \#\tilde{K}_{m,n} \mu (E_n) \biggl(\frac{t}{n} + \frac{c}{n^{\tau c_2}} \biggr)^{m} \biggl( 1+\sum_{p=1}^{m-1} c^p\frac{(\log n)^{p(1+2\alpha)}}{n^p} \Big( 1+\frac{nc}{tn^{\tau c_1}} \Big)^p \biggr).
  \end{align*}
  The sum over $p$ may be rewritten as
  \begin{equation*}
  		\sum_{p=1}^{m-1} \biggl (c\frac{(\log n)^{(1+2\alpha)}}{n}+\frac{c}{t}\frac{(\log n)^{(1+2\alpha)}}{n^{\tau c_1}} \biggr)^p\leq \sum_{p=1}^{\infty} \biggl (c\frac{(\log n)^{(1+2\alpha)}}{n^{\tau c_1}} \biggr)^p,
  \end{equation*} 
  since $\tau c_1<1$. For $n$ sufficiently large, this is a convergent
  geometric series, and we can bound it by its first summand to get
  \begin{align}
    \nonumber
    \sum_{k \in \tilde{K}_{m,n}} \int J_k \,
    \mathrm{d} \mu \leq \# \tilde{K}_{m,n} \mu
    (E_n) \biggl( \frac{t}{n} + \frac{c}{n^{\tau c_2}} \biggr)^m
    \biggl( 1 + c \frac{(\log n)^{1 + 2 \alpha}}{n^{\tau c_1}}
    \biggr) \\
    \leq \# \tilde{K}_{m,n} \mu(E_n)
    \biggl( \frac{c}{n} \biggr)^m \biggl( 1 + \frac{c}{n^{\tau c_1 - \varepsilon}} \biggr).
    \label{eq:expectationestimate}
  \end{align}
  To get a lower bound, we first note that it follows from $\tilde{K}_{m,n}=\cup_{p=0}^{m-1} \hat{K}_{m,n,p}$ and Lemma~\ref{lem:firstcounting} that for all $n$ sufficiently large,
  \[
    \frac{\# \hat{K}_{m,n,0}}{\# \tilde{K}_{m,n}} = 1 -
    \sum_{p = 1}^{m-1} \frac{\# \hat{K}_{m,n,p}}{\#
      \tilde{K}_{m,n}} \geq 1-  \sum_{p = 1}^{\infty} \bigg ( \frac{c(\log n)^{1+2\alpha}}{n} \bigg)^p \geq 1 - \frac{c}{n^{1-\varepsilon}}.
  \]
  Using \eqref{eq:Ikbounds} we then get that
  \begin{align*}
    \sum_{k \in \tilde{K}_{m,n}} \int J_k \,
    \mathrm{d} \mu & =\sum_{p=0}^{m-1}\sum_{k\in\hat{K}_{m,n,p}} \int J_k \,
                     \mathrm{d} \mu \geq \sum_{k\in\hat{K}_{m,n,0}} \int J_k \,
                     \mathrm{d} \mu\\
                   & \geq \# \hat{K}_{m,n,0} \mu
                     (E_n) \biggl( \frac{t}{n} - \frac{c}{n^{\tau c_2}}
                     \biggr)^m \\
                   &  \geq \# \tilde{K}_{m,n} \mu (E_n) \biggl( \frac{t}{n} -
                     \frac{c}{n^{\tau c_2}} \biggr)^m \biggl(1 -
                     \frac{c}{n^{1-\varepsilon}} \biggr)\\
                   & \geq \# \tilde{K}_{m,n} \mu (E_n) \biggl( \frac{c}{n}
                     \biggr)^m \biggl(1 -  \frac{c}{n^{1-\varepsilon}} \biggr)
  \end{align*}

  In conclusion, we have proved \eqref{eq:Ik-sum}.
\end{proof}

We move on to consider the second sum, i.e.\
$\sum\int J_k J_j \, \mathrm{d}\mu$. To obtain upper and lower
bounds, we follow a similar approach as in the case of the first
sum. The counting arguments are about the same, but the arguments
around are more complicated. The goal is to obtain the following
proposition.

\begin{proposition}
  \label{prop:JkJj-sum}
  We have for $n$ sufficiently large,
  \[
    1 - \frac{c}{n^{1-\varepsilon}} \leq \frac{\displaystyle
      \sum_{\substack{k \in \tilde{K}_{m_1,n} \\j \in
          \tilde{K}_{m_2,n} }} \int J_k J_j \, \mathrm{d} \mu}{
      \# \tilde{K}_{m_1,n} \# \tilde{K}_{m_2,n} \mu (E_n)^2
      \bigl(\frac{c}{n} \bigr)^{m_1 + m_2}} \leq 1 +
    \frac{c}{n^{\tau c_0 - \varepsilon}}.
  \]
\end{proposition}

Before we begin with the detailed proof of the proposition we give some intuition into the structure of $\int J_k J_j \, \mathrm{d}\mu$.
For the rest of the proof, suppose always that $k = (k_1,\ldots,k_{m_1})\in \tilde{K}_{m_1,n}$ and
$j = (j_1,\ldots,j_{m_2})\in \tilde{K}_{m_2,n}$, and fix the notation $\bar{m} := m_1 + m_2$. Consider the merged sequence of $k$ and $j$, that is $l := l(k,j) := (l_1, \ldots, l_{\bar{m}})$, the sequence containing all elements of $k$ and $j$ organized in increasing order. We have
\begin{align*}
  \int J_k J_j \mathrm{d} \mu
  &= \int \biggl(\int_{E_n} \ldots \, \mathrm{d}\mu(y) \biggr)
    \biggl(\int_{E_n} \ldots \, \mathrm{d}\mu(z) \biggr) \,
    \mathrm{d}\mu (x) \\
  &= \int_{E_n} \int_{E_n} \biggl( \int \mathbbm{1}_{B_{y \text{
    or } z}} (T^{l_1} x) \ldots \mathbbm{1}_{B_{y \text{ or } z}}
    (T^{l_{\bar{m}}} x) \, \mathrm{d} \mu(x) \biggr) \,
    \mathrm{d}\mu(y)\mathrm{d}\mu(z).  
\end{align*}
We will use exponential mixing on the inner integral, which we
denote by $H_l (y,z)$.  While we have reasonable separation (i.e.\ $c_1\log n$) between
consecutive numbers in $k$ and between consecutive numbers in
$j$, it is possible that a number in $j$ is very close to a
number in $k$. Consequently, consecutive numbers in $l$ might not be well-separated which complicates the application of exponential mixing.

We deal with this issue in the following way. For a sequence $s = (s_1,\ldots,s_{m})$ we will refer to the numbers $s_{i+1}-s_i$, $1\leq i\leq m-1$ as the \emph{gaps of $s$}. 
Let $l=l(k,j)\in K_{\bar{m},n}$ and denote by $l^g:=(l_{2}-l_1,\dots, l_{\bar{m}}-l_{\bar{m}-1})$ its \emph{associated gap sequence}. 

Let $c_0<\frac14 c_1$ and denote by $p_1$ the number of gaps in $l$ of size less than $c_0 \log n$. Note that for such a gap $l_{i+1}-l_i$, one of $l_{i+1}, l_i$ must have come from $k$ and one from $j$ and such gaps can not be consecutive elements in $l^g$. In particular this means that $l_{i+2}-l_i>c_1\log n$ is always true since the gaps in $k$ and $j$ have this property. Denote by $p_2$ the number of gaps in $l$ of size between $c_0\log n$ and $c_2\log n$. Finally, denote by $p_3$ the number of gaps in $l$ of size larger than $c_2\log n$.

We now use exponential mixing to split $H_l (y,z)$ into a product of integrals. 
We apply exponential mixing repeatedly in the same way as earlier with one exception. If the first gap of $l$ is not counted by $p_1$ we proceed as usual, considering the left-most characteristic function to be the $BV$-function while the remaining product represents the $L^1$-function. In the case when the first gap of $l$ is counted by $p_1$ we consider the product of the two left-most characteristic functions to be the $BV$-function while the remaining product represents the $L^1$-function. We point out that the assumption that $T$ has at most finitely many branches is used here to ensure the validity of this step. We see that it plays an important role whether a gap is preceded by a $p_1$-gap. Hence, denote by $p_3^*$ the number of $p_3$-gaps that immediately follow a $p_1$-gap and let $p_2^*$ be defined analogously.
Note that $p_1=p_2^*+p_3^*$. Repeated exponential mixing applied to $H_l(y,z)$ would normally produce a product of $\bar{m}$ factors but since the $p_1$-gaps join two functions into one, we only get $\bar{m}-p_1$ factors. Once $H_l(y,z)$ has been split into a product it will contain factors of the forms listed below.

The first case is that the factor comes from a gap $l_{i+1} - l_i > c_2 \log n$ that is not immediately preceeded by a $p_1$-gap. There will be $p_3-p_3^*$ such gaps. We can
estimate this factor by
\[
  \biggl( \frac{t}{n} - \frac{c}{n^{\tau c_2}} \biggr)
  \quad \text{or} \quad
  \biggl( \frac{t}{n} + \frac{c}{n^{\tau c_2}} \biggr)
\]
depending on if we want a lower or upper bound.

Gaps with $l_{i+1}-l_i$ between $c_0\log n$ and $c_2\log n$ that are not immediately preceeded by a $p_1$-gap results in factors that are estimated by
\[
0 \quad \text{or} \quad \biggl( \frac{t}{n} + \frac{c}{n^{\tau
		c_0}} \biggr) \leq \frac{c}{n^{\tau c_0}},
\]
depending on if we want a lower or upper bound. There will be $p_2-p_2^*$ many gaps of this type. 


For the $p_1$ many gaps of size $l_{i+1}-l_i<c_0\log n$ the situation is slightly more complicated. Suppose $p_1\geq 1$. For the lower bound we can again use 0. The factors are the sum of an integral and an error term and the error term depends on the size of $l_{i+2}-l_i$ which is easily seen when writing out explicitly the iterative application of exponential mixing described above. If $l_{i+2}-l_i$ is between $c_1\log n$ and $c_2 \log n$ we get
\begin{equation}
  \label{eq:complicatedfactor}
  \biggl( \int \mathbbm{1}_{B_{y \text{ or } z}} (T^{l_{i+1}} x_s)
  \mathbbm{1}_{B_{y \text{ or } z}} (T^{l_{i}} x_s) \, \mathrm{d}
  \mu (x_s) + B^{l_{i+1} - l_i} \frac{c}{n^{\tau c_1 }} \biggr),\enskip s=1,\dots,p_1
\end{equation}
while if $l_{i+2}-l_i>c_2\log n$ we get
\begin{equation}
  \label{eq:complicatedfactor2}
  \biggl( \int \mathbbm{1}_{B_{y \text{ or } z}} (T^{l_{i+1}} x_s)
  \mathbbm{1}_{B_{y \text{ or } z}} (T^{l_{i}} x_s) \, \mathrm{d}
  \mu (x_s) + B^{l_{i+1} - l_i} \frac{c}{n^{\tau c_2 }} \biggr), \enskip s=1,\dots,p_1
\end{equation}
for some $B>1$, where the term $B^{l_{i+1} - l_i}$ comes from the bounded variation
norm of the product inside the integral. Note that in \eqref{eq:complicatedfactor} and \eqref{eq:complicatedfactor2}, in the product of two $\mathbbm{1}_{B_{y \text{ or } z}}$-functions, one will be $\mathbbm{1}_{B_{y}}$ and the other will be $\mathbbm{1}_{B_{z}}$ since the $p_1$ gaps always consist of one element from $k$ and one from $j$. 

We first give an estimate on the error term. Since $l_{i+1} - l_i \leq c_0
\log n$, we have $B^{l_{i+1} - l_i} \leq B^{c_0 \log n} = n^{c_0
  \log B}$. Choosing $c_0$ small enough, we then have the bounds
\[
  B^{l_{i+1} - l_i} \frac{c}{n^{\tau c_1 }} \leq
  \frac{c}{n^{\tau c_1 / 2}}
\]
and
\[
  B^{l_{i+1} - l_i} \frac{c}{n^{\tau c_2 }} \leq
  \frac{c}{n^{\tau c_2 / 2}}.
\]
which will prove to be sufficient.

We now give an estimate on the integral in
\eqref{eq:complicatedfactor} and
\eqref{eq:complicatedfactor2}. We have that
\begin{equation}
  \label{eq:JkJj}
  \int J_k J_j \, \mathrm{d} \mu \leq \iint \prod \ldots \,
  \mathrm{d} \mu(y) \mathrm{d}\mu (z)
\end{equation}
where the product consists of factors of the kind described above. We expand
the product as a sum of products, and in each term of the
sum, we change order of integration. There are only the
terms containing factors of the form
\eqref{eq:complicatedfactor} and \eqref{eq:complicatedfactor2}
  that contain integrals and in which the order of integration
  can be changed. The product of integrals that appear in such a term has the form
\begin{multline*}
 Q:= \int \ldots \int \biggl( \int_{E_n} 
  \mathbbm{1}_{B_y} (T^{a_1} x_1) \ldots \mathbbm{1}_{B_y}
  (T^{a_r} x_r) \, \mathrm{d}\mu(y) \\
  \times \int_{E_n}  \mathbbm{1}_{B_z} (T^{b_1} x_1)
  \ldots \mathbbm{1}_{B_z} (T^{b_r} x_r)  \,
  \mathrm{d}\mu(z) \biggr) \, \mathrm{d}\mu(x_1) \ldots
  \mathrm{d}\mu(x_r),
\end{multline*}
with $1\leq r\leq p_1$. Here the $a_i$'s and $b_i$'s represent $l_i$'s that appear in the argument of the $\mathbbm{1}_{B_y}$ and $\mathbbm{1}_{B_z}$ respectively. For simplicity these have been re-enumerated to coincide with the index of the integration variable $x_s$. The two innermost integrals, over $y$ and $z$, are treated in the
same way and we demonstrate the computation in the $y$-case. Making use of \eqref{eq:Lebesgue} twice and denoting by $d(\cdot,\dots,\cdot)$ the maximal distance between the points
in question we get
\begin{align*}
  \int_{E_n}  \mathbbm{1}_{B_y} (T^{a_1} x_1) &\ldots
  \mathbbm{1}_{B_y} (T^{a_r} x_r)  \, \mathrm{d}\mu(y)\\
  &\leq \int  \mathbbm{1}_{B(T^{a_1} x_1,r_n(y,t))} (y) \ldots
    \mathbbm{1}_{B(T^{a_r} x_r,r_n(y,t))} (y)  \, \mathrm{d}\mu(y)\\
    &\leq \int  \mathbbm{1}_{B(T^{a_1} x_1,C\frac{t}{n})} (y) \ldots
    \mathbbm{1}_{B(T^{a_r} x_r,C\frac{t}{n})} (y)  \, \mathrm{d}\mu(y)\\
&\leq\mu \Big(B \Big(T^{a_1}x_1,C\frac{t}{n}\Big)\Big)\mathbbm{1}_{\{d (T^{a_1}
	x_1,\ldots, T^{a_r} x_r) < 2 C\frac{t}{n}\}}\\
  &\leq c \biggl( \frac{t}{n} \biggr) \mathbbm{1}_{\{d (T^{a_1}
  	x_1,\ldots, T^{a_r} x_r) < 2 C\frac{t}{n}\}}.
\end{align*}
From the second to the third line, the multiplication by $\mathbbm{1}_{\{d (T^{a_1}
	x_1,\ldots, T^{a_r} x_r) < 2 C\frac{t}{n}\}}$ does not change anything, while bounding $r-1$ characteristic functions by $1$ gives the inequality. The same
estimate holds for the integral over $z$ although in this case we simply bound all $r-1$ characteristic functions by 1 and use the upper bound $c \frac{t}{n}$.

Thus, integrating over $x_1, \ldots, x_r$ we obtain that
\begin{align*}
 Q &\leq c \biggl( \frac{t}{n}
 \biggr)^{2} \int \ldots \int \mathbbm{1}_{\{d (T^{a_1}
  	x_1,\ldots, T^{a_r} x_r) < 2 C\frac{t}{n}\}} \, \mathrm{d}\mu(x_1) \ldots \mathrm{d}\mu(x_r) \\
  	&\leq  c \biggl( \frac{t}{n}
  	\biggr)^{2} \int \ldots \int \mathbbm{1}_{\{ \max_{2\leq i \leq r} d (T^{a_1}
  		x_1, T^{a_i} x_i) < 2 C\frac{t}{n}\}} \, \mathrm{d}\mu(x_1) \ldots \mathrm{d}\mu(x_r)\\
  		&=c \biggl( \frac{t}{n}
  		\biggr)^{2} \int \ldots \int \prod_{i=2}^r \mathbbm{1}_{B(T^{a_1}x_1, 2 C\frac{t}{n})}(T^{a_i}x_i) \, \mathrm{d}\mu(x_1) \ldots \mathrm{d}\mu(x_r)\\
 & \leq c \biggl( \frac{c}{n} \biggr)^{2 + (r-1)} .
\end{align*}
Here we used \eqref{eq:Lebesgue} again. This shows that after integrating over $y$ and $z$, the integrals given in
\eqref{eq:complicatedfactor} and \eqref{eq:complicatedfactor2} contribute at most $cn^{-(1+\frac{1}{r})}$ in the products described above, and since $r\leq p_1$ we get the general bound $cn^{-(1+\frac{1}{p_1})}$ on their contribution. Factorizing the sum of products back into one product, we see that we can bound the factors from \eqref{eq:complicatedfactor} and \eqref{eq:complicatedfactor2} by
\begin{equation*}
  \biggl(\frac{c}{n^{1 + \frac{1}{p_1}}} + \frac{c}{n^{\tau c_1
      \frac{1}{2}}} \biggr)\leq \frac{c}{n^{\tau c_1
      \frac{1}{2}}}
\end{equation*}
and
\begin{equation*}
	\biggl(\frac{c}{n^{1 + \frac{1}{p_1}}} + \frac{c}{n^{\tau c_2
			\frac{1}{2}}} \biggr)=\biggl(\frac{c}{n} + \frac{c}{n^{\tau c_2
			\frac{1}{2}-\frac{1}{p_1}}} \biggr)\frac{1}{n^{\frac{1}{p_1}}}\leq \biggl(\frac{c}{n} + \frac{c}{n^{\tau c_2
			\frac{1}{2}-1}} \biggr)\frac{c}{n^{\frac{1}{p_1}}}.
\end{equation*}
Set $\delta_2:=\tau c_2\frac{1}{2}-1$ and note that $\delta_2>1$. There are $p_2^*$ many instances of \eqref{eq:complicatedfactor} and $p_3^*$ instances of \eqref{eq:complicatedfactor2} in the product. We estimate the total contribution of \eqref{eq:complicatedfactor} and \eqref{eq:complicatedfactor2} as follows.
\begin{align*}
	\biggl ( \frac{c}{n^{\tau c_1 \frac{1}{2}}}&\biggr )^{p_2^*}\biggl (\biggl(\frac{c}{n} + \frac{c}{n^{\delta_2}} \biggr)\frac{c}{n^{\frac{1}{p_1}}}\biggr )^{p_3^*}\\
	&=\biggl ( \frac{c}{n^{\tau c_1 \frac{1}{4}}}\biggr )^{p_2^*}\biggl(\frac{c}{n} + \frac{c}{n^{\delta_2}} \biggr)^{p_3^*}     \biggl [ \biggl ( \frac{c}{n^{\tau c_1 \frac{1}{4}}}\biggr )^{p_2^*} \biggl (\frac{c}{n^{\frac{1}{p_1}}}\biggr)^{p_3^*}\biggr ]
\end{align*}
For the factor in the square bracket we consider two cases. Suppose $p_3^*=p_1$. Then $p_2^*=0$ and
\begin{equation*}
	 \biggl ( \frac{c}{n^{\tau c_1 \frac{1}{4}}}\biggr )^{p_2^*} \biggl (\frac{c}{n^{\frac{1}{p_1}}}\biggr)^{p_3^*}=\frac{c}{n}\leq \frac{c}{n^{\tau c_1 \frac{1}{4}}}.
\end{equation*}
 Suppose $p_3^*<p_1$. Then $p_2^*\geq 1$ and for $n$ sufficiently large,
 \begin{equation*}
 	 \biggl ( \frac{c}{n^{\tau c_1 \frac{1}{4}}}\biggr )^{p_2^*} \biggl (\frac{c}{n^{\frac{1}{p_1}}}\biggr)^{p_3^*}\leq  \frac{c}{n^{\tau c_1 \frac{1}{4}}} (1)^{p_3^*}\leq \frac{c}{n^{\tau c_1 \frac{1}{4}}}.
 \end{equation*}
So in either case we get the bound $\frac{c}{n^{\tau c_1 \frac{1}{4}}}$ and the estimate of the total contribution coming from \eqref{eq:complicatedfactor} and \eqref{eq:complicatedfactor2} becomes
\begin{equation*}
	\biggl ( \frac{c}{n^{\tau c_1 \frac{1}{4}}}\biggr )^{p_2^*}\biggl(\frac{c}{n} + \frac{c}{n^{\delta_2}} \biggr)^{p_3^*}\frac{c}{n^{\tau c_1 \frac{1}{4}}}.
\end{equation*}



We can now collect our observations to write up our estimates of $\int J_k J_j \,
\mathrm{d}\mu$. Note that going forward we will for simplicity convert all instances of $t$ into the generic constant $c$.  If $p_1 = p_2 = 0$, then we have
\begin{equation}
  \label{eq:JkJj-goodcase}
  \biggl( \frac{c}{n} - \frac{c}{n^{\tau c_2}} \biggr)^{\bar{m}}
  \mu (E_n)^2 \leq
  \int J_k J_j \, \mathrm{d}\mu \leq \biggl( \frac{c}{n} +
  \frac{c}{n^{\tau c_2}} \biggr)^{\bar{m}} \mu (E_n)^2.
\end{equation}

When $(p_1,p_2) \neq (0,0)$, we only need an upper
bound (and the trivial lower bound that $\int J_k J_j \, \mathrm{d}\mu$ is non-negative). Note that 
\begin{align*}
		p_3-p_3^*&=\bar{m}-2p_1 -(p_2-p_2^*)\\
		&=\bar{m}-p_1-p_2-p_3^*.
	\end{align*}



Collecting all of the above estimates we get for $p_1 > 0$
\begin{align}
  \label{eq:JkJj-badcase1}
  \int J_k J_j \, \mathrm{d}\mu &\leq \biggl(\frac{c}{n} + \frac{c}{n^{\tau c_2}}
  \biggr)^{p_3-p_3^*}  \biggl( \frac{c}{n^{\tau
      c_0}}\biggr)^{p_2-p_2^*}
   \biggl ( \frac{c}{n^{\tau c_1 \frac{1}{4}}}\biggr )^{p_2^*}\biggl(\frac{c}{n} + \frac{c}{n^{\delta_2}} \biggr)^{p_3^*}\frac{c}{n^{\tau c_1 \frac{1}{4}}}\nonumber\\
  &\leq \biggl(
  \frac{c}{n} + \frac{c}{n^{\delta_2}}
  \biggr)^{\bar{m}-p_1-p_2}  \biggl( \frac{c}{n^{\tau
  		c_0}}\biggr)^{p_2} \frac{c}{n^{\tau c_1 \frac{1}{4}}}
\end{align}  
and
\begin{equation}
  \label{eq:JkJj-badcase2}
  \int J_k J_j \, \mathrm{d}\mu \leq \biggl(
  \frac{c}{n} + \frac{c}{n^{\tau c_2}} \biggr)^{\bar{m}-p_2} \biggl( \frac{c}{n^{\tau c_0}}\biggr)^{p_2},
\end{equation}
when $p_1 = 0$. Note that in both \eqref{eq:JkJj-goodcase} and \eqref{eq:JkJj-badcase2} we can also replace $\tau c_2$ with $\delta_2$.

We will need to count how many pairs of sequences $k$ and $j$
there are for fixed numbers $p_1, p_2$. Note that fixing $p_1$ and $p_2$ uniquely determines $p_3$ when $\bar{m}$ is given. For this purpose we again need to introduce some notation. Let $\bar{K}_{m_1,m_2,n}:=\tilde{K}_{m_1,n}\times \tilde{K}_{m_2,n}$. In the following we will also be thinking of $l$ as the map $l:\bar{K}_{m_1,m_2,n}\to K_{\bar{m},n}$, $(k,j)\mapsto l(k,j)$, that is, the map which takes a pair of sequences and merges them as described earlier. Define $L_{m_1,m_2,n}:=l(\bar{K}_{m_1,m_2,n})$, i.e.\ the image of $\bar{K}_{m_1,m_2,n}$ under $l$. Suppose a given $l\in L_{m_1,m_2,n}$ has $p_1$ and $p_2$ many gaps of the types introduced earlier. We call the vector $g_p:=(p_1,p_2)$ the \emph{gap profile} of $l$ and define
\[
	L_{m_1,m_2,n,g_p}:=\{ l \in L_{m_1,m_2,n} \,:\, l \text{ has gap profile } g_p \}
\]
Similarily, define
\begin{align*}
	\bar{K}_{m_1,m_2,n,g_p}:&=l^{-1}(L_{m_1,m_2,n,g_p})\\
	&=\{ (k,j) \in \bar{K}_{m_1,m_2,n} \,:\, l(k,j) \text{ has gap profile } g_p \}.
\end{align*}
Henceforth, let $P = p_1 + p_2$.


\begin{lemma}
  \label{lem:secondcounting}
  We have
  \[
    \frac{\# \bar{K}_{m_1,m_2,n,g_p} }{ \# \tilde{K}_{m_1,n}
      \# \tilde{K}_{m_2,n} } \leq 2^{\bar{m}} c^{P} \frac{(\log
      n)^{P(1 + 2 \alpha)}}{ n^{P}}.
  \]
\end{lemma}

\begin{proof}
We first consider an upper bound on the number of pairs $(k,j)\in \bar{K}_{m_1,m_2,n,g_p}$ that merge to the same element $l\in L_{m_1,m_2,n,g_p}$. Given an $l$ and ignoring all gap restrictions, each element in $l$ has two options, namely stemming from $k$ or $j$. Since $l$ contains $\bar{m}$ elements, this gives the bound $\# \bar{K}_{m_1,m_2,n,g_p}\leq 2^{\bar{m}} \# L_{m_1,m_2,n,g_p}$. Recall also the general inequality $\binom{n}{m_1+m_2}\leq \binom{n}{m_1}\binom{n}{m_2}$. We then get
\begin{align*}
	\frac{\# \bar{K}_{m_1,m_2,n,g_p} }{ \# \tilde{K}_{m_1,n} \# \tilde{K}_{m_2,n} } &\leq \frac{2^{\bar{m}}\#L_{m_1,m_2,n,g_p}}{\binom{n}{\bar{m}}} \frac{\binom{n}{\bar{m}}}{\# \tilde{K}_{m_1,n} \# \tilde{K}_{m_2,n}}\\
	&\leq \frac{2^{\bar{m}}\#L_{m_1,m_2,n,g_p}}{\binom{n}{\bar{m}}} \frac{\binom{n}{m_1}\binom{n}{m_2}}{\# \tilde{K}_{m_1,n} \# \tilde{K}_{m_2,n}}\\
	&\leq c\frac{2^{\bar{m}}\#L_{m_1,m_2,n,g_p}}{\binom{n}{\bar{m}}}
\end{align*}
by Lemma~\ref{lem:binomial}. The rest of the proof is argued like in Lemma~\ref{lem:firstcounting}, that is, $\bar{m}-P$ elements are distributed freely among $n$ options and on the remaining $P$ we only impose the restriction that they must be within $c_2\log n$ of another element while we pretend that all $P$ elements have $\bar{m}$ neighborhoods in which they can be located. Repeating this argument gives the conclusion of the lemma. 
\end{proof}

We are now ready to prove Proposition~\ref{prop:JkJj-sum}.

\begin{proof}[Proof of Proposition~\ref{prop:JkJj-sum}]
  We start with the upper bound. Let $A:=\{0,\dots,\bar{m}-1\}^2$ and let $A_g\subset A$ denote the set of admissible gap profiles in $L_{m_1,m_2,n}$. Write $\bar{0}=(0,0)$. Denote the upper bound in \eqref{eq:JkJj-goodcase} by $B_1$ and the bounds in \eqref{eq:JkJj-badcase1} and \eqref{eq:JkJj-badcase2} by $B_2$ and $B_3$ respectively. The estimate is structured as follows
  \begin{align*}
  	 \sum_{\substack{k \in \tilde{K}_{m_1,n} \\j \in
  			\tilde{K}_{m_2,n} }}& \int J_k J_j \, \mathrm{d} \mu \\
  		&=  \sum_{g_p\in A_g}\sum_{(k,j) \in \bar{K}_{m_1, m_2,n, g_p}} \int J_k J_j \, \mathrm{d} \mu\\
  		&= \sum_{(k,j) \in \bar{K}_{m_1, m_2,n, \bar{0}}} \int J_k J_j \, \mathrm{d} \mu\\ 
  		&\hspace{1.5cm}+\sum_{\substack{g_p\in A_g\\p_1>0}}\sum_{(k,j) \in \bar{K}_{m_1, m_2,n, g_p}} \int J_k J_j \, \mathrm{d} \mu\\
  		&\hspace{3cm}+\sum_{\substack{g_p\in A_g\\p_1=0, p_2>0}}\sum_{(k,j) \in \bar{K}_{m_1, m_2,n, g_p}} \int J_k J_j \, \mathrm{d} \mu\\
  		&\leq \sum_{(k,j) \in \bar{K}_{m_1, m_2,n, \bar{0}}} B_1\\ 
  		&\hspace{1.5cm}+ \sum_{\substack{g_p\in A_g\\p_1>0}}\sum_{(k,j) \in \bar{K}_{m_1, m_2,n, g_p}} B_2\\
  		&\hspace{3cm}+\sum_{\substack{g_p\in A_g\\p_1=0, p_2>0}}\sum_{(k,j) \in \bar{K}_{m_1, m_2,n, g_p}} B_3\\
  		&= \#\bar{K}_{m_1, m_2,n, \bar{0}} B_1\\ 
  		&\hspace{1.5cm}+ \sum_{\substack{g_p\in A_g\\p_1>0}}\# \bar{K}_{m_1, m_2,n, g_p} B_2\\
  		&\hspace{3cm}+\sum_{\substack{g_p\in A_g\\p_1=0, p_2>0}} \# \bar{K}_{m_1, m_2,n, g_p} B_3\\
  		&\leq \#\bar{K}_{m_1, m_2,n} 2^{\bar{m}} B_1\\ 
  		&\hspace{1.5cm}+ \sum_{\substack{g_p\in A_g\\p_1>0}}\# \bar{K}_{m_1, m_2,n} 2^{\bar{m}} c^{P} \frac{(\log
  			n)^{P(1 + 2 \alpha)}}{ n^{P}}B_2\\
  		&\hspace{3cm}+\sum_{\substack{g_p\in A_g\\p_1=0, p_2>0}}\# \bar{K}_{m_1, m_2,n} 2^{\bar{m}} c^{p_2} \frac{(\log
  			n)^{p_2(1 + 2 \alpha)}}{ n^{p_2}} B_3
  \end{align*}
  by Lemma~\ref{lem:secondcounting}. Factoring out $\bigl(\frac{c}{n} + \frac{c}{n^{\delta_2}}
  \bigr)^{\bar{m}}$  in $B_1, B_2$ and $B_3$ and inserting the estimates from \eqref{eq:JkJj-goodcase}, \eqref{eq:JkJj-badcase1} and \eqref{eq:JkJj-badcase2} we get the total estimate
  \begin{multline*}
    \sum_{\substack{k \in \tilde{K}_{m_1,n} \\j \in
        \tilde{K}_{m_2,n} }} \int J_k J_j \, \mathrm{d} \mu \leq
    \#\bar{K}_{m_1, m_2,n} 2^{\bar{m}} \biggl(\frac{c}{n} + \frac{c}{n^{\delta_2}} \biggr)^{\bar{m}} \biggl[\mu (E_n)^2 + \\
    + \sum_{\substack{g_p\in A_g\\p_1>0}} c^{P} \frac{(\log n)^{P (1 + 2 \alpha)}}{ n^{P}} \biggl(\frac{c}{n}\biggr)^{-P}  \biggl(
    \frac{c}{n^{\tau c_0}} \biggr)^{p_2} \frac{c}{n^{\tau c_1 \frac{1}{4}}}\\
    + \sum_{\substack{g_p\in A_g\\p_1=0, p_2>0}} c^{p_2} \frac{(\log n)^{p_2 (1 + 2\alpha)}}{ n^{p_2}} \biggl(\frac{c}{n}\biggr)^{-p_2} \biggl(
    \frac{c}{n^{\tau c_0}} \biggr)^{p_2}  \biggr],
  \end{multline*}
	Here, it is only the sum over $p_1$ which requires further analysis. The sum
  	over $p_1$, taking out all factors depending on $p_2$,
  	is bounded by
  	\[
  	\sum_{p_1 = 1}^{M_n} c^{p_1}(\log n)^{p_1 (1 + 2 \alpha)}
  	\frac{1}{n^{\tau c_1 \frac{1}{4}}}.
  	\]
  	This is a geometric series with quotient
  	$(\log n)^{(1 + 2 \alpha)}$, and hence
  	\begin{multline*}
  		\sum_{p_1 = 1}^{M_n} c^{p_1}(\log n)^{p_1 (1 + 2 \alpha)}
  		\frac{1}{n^{\tau c_1 \frac{1}{4}}} \leq c^{M_n} \frac{(\log
  			n)^{M_n (1 + 2 \alpha)}}{n^{\tau c_1 \frac{1}{4}}} \leq c^{M_n} \frac{(\log
  			n)^{(\log
  				n)^\alpha (1 + 2 \alpha)}}{n^{\tau c_1 \frac{1}{4}}} \\
  		= \frac{c^{M_n}}{n^{\tau c_1 \frac{1}{4}}} \exp [ (\log
  		n)^\alpha (1 + 2 \alpha) \log \log n ] \leq \frac{c}{n^{{\tau
  					c_1 \frac{1}{4}} - \varepsilon}},
  \end{multline*}
for any arbitrarily small $\varepsilon>0$, where we used
\[
\frac{\exp [ (\log
  		n)^\alpha (1 + 2 \alpha) \log \log n ]}{n^{\varepsilon}} \to 0
\]
in the last step. For the record, the sum over $p_2$ is estimated by
  \[
    \sum_{p_2 = 1}^\infty c^{p_2} \frac{(\log n)^{p_2 (1 +
        \alpha)}}{n^{p_2 \tau c_0}} \leq c \frac{(\log n)^{1 +
        \alpha}}{n^{\tau c_0}},
  \]
 Hence, the (estimate of the) contribution from the $p_2$-sum is the worse
  (since we assumed $c_0 <  \frac{1}{4}c_1$), and we get
  in total using $\delta_2 >1$ that
  \begin{align*}
    \sum_{\substack{k \in \tilde{K}_{m_1,n} \\j \in
    \tilde{K}_{m_2,n} }} \int J_k J_j \, \mathrm{d} \mu
    & \leq \#\bar{K}_{m_1, m_2,n} 2^{\bar{m}} \biggl(\frac{c}{n} + \frac{c}{n^{\delta_2}}
      \biggr)^{\bar{m}} \biggl(\mu(E_n)^2 + \frac{c (\log n)^{1 +
      \alpha}}{n^{\tau c_0}} \biggr) \\
    & \leq \#\bar{K}_{m_1, m_2,n}  \biggl(\frac{c}{n} \biggr)^{\bar{m}}
      \biggl(1 + \frac{c }{n^{\delta_2-1}} \biggr)^{\bar{m}} \biggl(
      \mu(E_n)^2 + \frac{c (\log n)^{1 + \alpha}}{n^{\tau c_0}}
      \biggr) \\
    & \leq \#\bar{K}_{m_1, m_2,n}  \biggl(\frac{c}{n} \biggr)^{\bar{m}}
      \biggl( \mu(E_n)^2 + \frac{c
      (\log n)^{1 + \alpha}}{n^{\tau c_0}} \biggr).
  \end{align*}
  Hence, for any $\varepsilon > 0$, we have
  \begin{equation}
    \label{eq:correlationestimate}
    \sum_{\substack{k \in \tilde{K}_{m_1,n} \\j \in
        \tilde{K}_{m_2,n} }} \int J_k J_j \, \mathrm{d} \mu \leq
    \biggl(\mu(E_n)^2 + \frac{c}{n^{\tau c_0 - \varepsilon}}
    \biggr) \# \tilde{K}_{m_1,n} \# \tilde{K}_{m_2,n}
    \biggl(\frac{c}{n} \biggr)^{\bar{m}},
  \end{equation}
  which proves the upper bound in
  Proposition~\ref{prop:JkJj-sum}.

  We now consider the lower bound. Recall that $\bar{K}_{m_1,m_2,n}=\cup_{g_p\in A_g} \bar{K}_{m_1,m_2,n,g_p}$ where the union is disjoint. We then have by
  Lemma~\ref{lem:secondcounting} that for $n$ sufficiently large
  \begin{align*}
    \frac{\#\bar{K}_{m_1,m_2,n,\bar{0}}}{\# \bar{K}_{m_1,m_2,n}}&=1-\sum_{\substack{g_p\in A_g\\g_p\neq \bar{0}}}\frac{\#\bar{K}_{m_1,m_2,n,g_p}}{\#\bar{K}_{m_1,m_2,n}} \\
    &\geq 1 - 2^{\bar{m}}\sum_{\substack{g_p\in A_g\\g_p\neq \bar{0}}} c^P\frac{(\log n)^{P(1+2\alpha)}}{n^P}\\
    &\geq 1 - 2^{\bar{m}}2\biggl( \sum_{p_1=0}^{M_n} c^{p_1}\frac{(\log n)^{p_1(1+2\alpha)}}{n^{p_1}}\biggr)\biggl (\sum_{p_2=1}^{M_n} c^{p_2}\frac{(\log n)^{p_2(1+2\alpha)}}{n^{p_2}}\biggr )\\
    &\geq 1 - 2^{\bar{m}} c\frac{(\log n)^{(1+2\alpha)}}{n}.
  \end{align*}
   We thereby get that for any $\varepsilon>0$,
  \[
    \frac{\#\bar{K}_{m_1,m_2,n,\bar{0}}}{\# \bar{K}_{m_1,m_2,n}} \geq 1 - 2^{\bar{m}}c \frac{(\log
      n)^{1+2\alpha}}{n} \geq 1 - \frac{c}{n^{1-\varepsilon}}.
  \]
  By decomposing the sum as in the beginning of the proof and inserting the lower bound from \eqref{eq:JkJj-goodcase} as well as the trivial estimate
  $\int J_k J_j \, \mathrm{d} \mu \geq 0$ we get
  that
  \begin{align*}
    \sum_{\substack{k \in \tilde{K}_{m_1,n} \\j \in
    \tilde{K}_{m_2,n} }} \int J_k J_j \, \mathrm{d} \mu &\geq \mu (E_n)^2 \biggl( \frac{c}{n} - \frac{c}{n^{\tau
    c_2}} \biggr)^{\bar{m}} \#\bar{K}_{m_1,m_2,n,\bar{0}} \\
    & \geq \mu (E_n)^2 \biggl( \frac{c}{n} - \frac{c}{n^{\tau
      c_2}} \biggr)^{\bar{m}} \biggl( 1 -
      \frac{c}{n^{1-\varepsilon}} \biggr) \# \tilde{K}_{m_1,n} \#
      \tilde{K}_{m_2,n} \\
    & \geq \mu (E_n)^2 \biggl( \frac{c}{n} \biggr)^{\bar{m}}
      \biggl(1 - \frac{c}{n^{1-\varepsilon}} \biggr) \#
      \tilde{K}_{m_1,n} \# \tilde{K}_{m_2,n},
  \end{align*}
  which is the lower bound of Proposition~\ref{prop:JkJj-sum}.
\end{proof}

We can now finish the proof by giving an estimate on
$\Lambda_n$. Recall from \eqref{eq:numlambda} that
\begin{align*}
	&\Lambda_n	 \leq  \sum_{m_1,m_2=1}^{M_n} \bigg\vert\#\tilde{K}_{m_1,n} \#\tilde{K}_{m_2,n}
	\mu(E_n)^2 \biggl( \frac{c}{n} \biggr)^{\bar{m}}+ \sum_{\substack{k \in \tilde{K}_{m_1,n} \\j \in \tilde{K}_{m_2,n} }}\int J_k J_j \,
	\mathrm{d} \mu\nonumber\\
	& \hspace{3cm} - \mu(E_n) \#\tilde{K}_{m_2,n}\biggl( \frac{c}{n} \biggr)^{m_2}  \sum_{k\in \tilde{K}_{m_1,n}} \int J_k \,
	\mathrm{d} \mu \\
	& \hspace{3cm}- \mu(E_n) \#\tilde{K}_{m_1,n}\biggl(\frac{c}{n} \biggr)^{m_1} \sum_{j\in \tilde{K}_{m_2,n}} \int J_j \,
	\mathrm{d} \mu \bigg \vert.\nonumber
\end{align*}
From Propositions~\ref{prop:Ik-sum} and \ref{prop:JkJj-sum} we see that for some $e_1\in [- \frac{c}{n^{1-\varepsilon}}, \frac{c}{n^{\tau c_0-\varepsilon}}]$ and $e_2, e_3\in [- \frac{c}{n^{1-\varepsilon}}, \frac{c}{n^{\tau c_1-\varepsilon}}]$, we have
\begin{align*}
	\Lambda_n	 &\leq  \sum_{m_1,m_2=1}^{M_n} \#\tilde{K}_{m_1,n} \#\tilde{K}_{m_2,n}
	\mu(E_n)^2 \biggl( \frac{c}{n} \biggr)^{\bar{m}} \big\vert 1+(1+e_1)-(1+e_2)-(1+e_3)\big \vert\\
	&=\sum_{m_1,m_2=1}^{M_n} \#\tilde{K}_{m_1,n} \#\tilde{K}_{m_2,n}
	\mu(E_n)^2 \biggl( \frac{c}{n} \biggr)^{\bar{m}} \big\vert e_1 -(e_2+e_3)\big \vert\\
	&\leq \sum_{m_1,m_2=1}^{M_n} \#\tilde{K}_{m_1,n} \#\tilde{K}_{m_2,n}
	\mu(E_n)^2 \biggl( \frac{c}{n} \biggr)^{\bar{m}} \frac{c}{n^{\tau c_0-\varepsilon}}\\
	&\leq \frac{c}{n^{\tau c_0-\varepsilon}}
\end{align*}
since the sums over $m_1$ and $m_2$ are convergent.  This means that for all sufficiently large $n$
\[
  \Lambda_n \leq \frac{c}{n^s},
\]
for some $c > 0$ and $s = \tau c_0 - \varepsilon$. This completes the proof of Theorem~\ref{the:vts} by Proposition~\ref{prop:crit}.

\section{Proof of Theorem~\ref{the:vts_extension}}
\label{sec:vtsproof_extension}

In Theorem~\ref{the:vts} we showed that there is a set $G \subseteq [0,1]$ with $\mu(G)=1$ such that for all $x \in G$,
  \[
    \lim_{n \to \infty} \mu \{\, y : \tau_{B(y,r_n(y,t))} (x) >
    t/\mu(B(y,r_n(y,t))) \,\} = e^{-t},
  \]
  for all $t \ge 0$, where $r_n(\cdot, t)$ was chosen such that $\mu (B(y,r_n(y,t))) = t/n$. In this section, we want to generalize this to arbitrary sequences $(r_n)$ with $r_n \searrow 0$ to conclude Theorem~\ref{the:vts_extension}.

Fix $t>0$. Let $(r_n)$ be any sequence of positive reals with
$r_n \searrow 0$. We show for every $x \in G$:
\[
  \forall \varepsilon>0 \ \exists N \in \N \ \forall m \geq N:
  \abs{ \mu \left(\Meng{y}{\mu(B(y,r_m))\tau_{B(y,r_m)}(x) \geq
        t}\right)-\ee^{-t}}< \e.
\]  
Let $x \in G$ and $\e>0$. By continuity of $f(z)=\ee^{-z}$ there
is $\tilde{\delta} \in (0,1)$ such that for all
$\delta \in (0,\tilde{\delta})$:
\begin{equation}\label{eq:cont}
  \bigl| \ee^{-t(1\pm \delta)} - \ee^{-t} \bigr| < \e/8.
\end{equation}
Take $\delta \in (0, \tilde{\delta}/2)$. As in the main proof let
$r_n(y,s)$ such that $\mu(B(y,r_n(y,s)))=s/n$.

By the main theorem, there exists $N_{x,\e}$ such that for all
$n \geq N_{x,\e}$,
\[
  \bigl| \mu (\{\, y : \mu(B(y,r_n(y, t(1-\delta))))
  \tau_{B(y,r_n(y, t(1-\delta)))} (x) \geq t(1-\delta) \, \}
  )-\ee^{-t(1-\delta)} \bigr| < \frac{\e}{8},
\]
and there is $\tilde{N}_{x,\e}$ such that for all
$n \geq \tilde{N}_{x,\e}:$
\begin{multline*}
  \bigl| \mu ( \{\, y : \mu(B(y,r_n(y,
        t(1+2\delta))))\tau_{B(y,r_n(y, t(1+2\delta)))}(x) \geq
        t(1+2\delta) \, \} ) \\
        - \ee^{-t(1+2\delta)}  \bigr| < \frac{\e}{8}.
\end{multline*}
Let $N_{x,\e, \dd} \geq \max \{N_{x,\e},\tilde{N}_{x,\e}\}$ such
that
\begin{equation}
  \frac{n}{n+1} \geq 1-\dd \ \text{ for all } n \geq N_{x,\e,\dd}.
\end{equation}
For all $y$ and for all $m$ there is $n_m(y)$ such that
\begin{equation} \label{eq:order} r_{n_m(y)+1}(y,t(1-\dd)) \leq
  r_m < r_{n_m(y)}(y,t(1-\dd)).
\end{equation}
We have $\lim_{m \to \infty} \frac{1}{n_m(y)}=0$ pointwise for
all $y$. By Egorov's Theorem there exists $B \subset [0,1]$ with
$\mu(B)< \e/8$ and $M \in \N$ such that for all $m \geq M$:
\[
  \frac{1}{n_m(y)} < \frac{1}{N_{x,\e,\dd}} \ \text{ for all } y
  \in [0,1] \setminus B.
\]

Similarly, for all $y$ and for all $m$ there is $\ell_m(y)$ such
that
\begin{equation}
  \label{eq:order2}
  r_{\ell_m(y)+1}(y,t(1+2\dd)) \leq r_m < r_{\ell_m(y)}(y,t(1+2\dd))
\end{equation}
and $\lim_{m \to \infty} \frac{1}{\ell_m(y)}=0$ pointwise for all
$y$. By Egorov's Theorem there exists $\tilde{B} \subset [0,1]$
with $\mu(\tilde{B})< \e/8$ and $\tilde{M} \in \N$ such that for
all $m \geq \tilde{M}$:
\[
  \frac{1}{\ell_m(y)} < \frac{1}{N_{x,\e,\dd}} \ \text{ for all }
  y \in [0,1] \setminus \tilde{B}.
\]

Take $m \geq \max \{ M, \tilde{M} \}$.

Then we have by \eqref{eq:order} that
\begin{align*}
  \mu & ( \{\, y \in [0,1] : \mu(B(y,r_m)) \tau_{B(y,r_m)}(x)
        \geq t \,\} ) \\
      & \leq \mu (\{\, y \in [0,1]\setminus
        B : \mu(B(y,r_m))\tau_{B(y,r_m)}(x) \geq t \,\} ) +
        \mu(B) \\
      & \leq \mu (\{\, y \in [0,1]\setminus
        B : \mu(B(y,r_{n_m(y)}(y,t(1-\dd))))
        \tau_{B(y,r_{n_{m}(y)+1}(y,t(1-\dd)))}(x)
        \geq t \,\} ) \\
      & \hspace{11cm} + \mu(B) \\
      & = \mu \Bigl( \Bigl\{ \, y \in [0,1]\setminus
        B : \mu(B(y,r_{n_{m}(y)+1}(y,t(1-\dd))))
        \tau_{B(y,r_{n_{m}(y)+1}(y,t(1-\dd)))}(x) \\
      & \hspace{8.6cm} \geq {\textstyle \frac{n_m(y)}{n_m(y)+1}t}
        \,\Bigr\} \Bigr) + \mu(B).
\end{align*}
Since $n_m(y) \geq N_{x,\e,\dd}$ for $y \in [0,1]\setminus B$, we
have $\frac{n_m(y)}{n_m(y)+1} \geq 1-\dd $. Thus, we conclude
\begin{align*}
  \mu & ( \{\, y \in [0,1] : \mu(B(y,r_m))\tau_{B(y,r_m)}(x)
    \geq t \,\} ) \\
  & \leq \mu \Bigl( \Bigl \{\, y \in [0,1]\setminus
         B : \mu(B(y,r_{n_{m}(y)+1}(y,t(1-\dd))))
    \tau_{B(y,r_{n_{m}(y)+1}(y,t(1-\dd)))}(x) \\
  & \hspace{8.6cm} \geq (1-\dd)t \, \Bigr\} \Bigr) + \mu(B) \\
  & \leq \mu \Bigl( \Bigl \{\, y \in [0,1] : \mu(B(y,r_{n_{m}(y)+1}(y,t(1-\dd))))
    \tau_{B(y,r_{n_{m}(y)+1}(y,t(1-\dd)))}(x) \\
  & \hspace{8.6cm} \geq (1-\dd)t \, \Bigr\} \Bigr) + \mu(B) \\
  & \leq \ee^{-t \cdot (1-\dd)} + \frac{\e}{8}+\frac{\e}{8} <
         \ee^{-t} + \frac{3\e}{8},
\end{align*}
where we used equation \eqref{eq:cont} in the last estimate.

For the converse direction, we estimate with \eqref{eq:order2}
that
\begin{align*}
  \mu & (\{\, y \in [0,1] : \mu(B(y,r_m))\tau_{B(y,r_m)}(x)
        \geq t \,\} ) \\
      & \geq \mu (\{\, y \in [0,1]\setminus
        \tilde{B} : \mu(B(y,r_m))\tau_{B(y,r_m)}(x) \geq
        t \,\} )  \\
      & \geq \mu \bigl( \bigl\{ \, y \in [0,1]\setminus
        \tilde{B} : \mu(B(y,r_{\ell_m(y)+1}(y,t(1+2\dd))))
        \tau_{B(y,r_{\ell_{m}(y)}(y,t(1+2\dd)))}(x) \\
      & \hspace{11cm} \geq t \, \bigr\} \bigr) \\
      & = \mu \Bigl ( \Bigl\{ \, y \in [0,1]\setminus
        \tilde{B} : \mu(B(y,r_{\ell_{m}(y)}(y,t(1+2\dd))))
        \tau_{B(y,r_{\ell_{m}(y)}(y,t(1+2\dd)))}(x) \\
      & \hspace{9cm} \geq \frac{\ell_m(y)+1}{\ell_m(y)}t \,
        \Bigr\} \Bigr) .
\end{align*}
Since $\ell_m(y) \geq N_{x,\e,\dd}$ for
$y \in [0,1]\setminus \tilde{B}$, we have
\[
  \frac{\ell_m(y)+1}{\ell_m(y)} \leq \frac{1}{1-\delta} = 1 +
  \frac{ \delta}{1-\delta} \leq 1+2\dd.
\]
Thus we conclude
\begin{align*}
  \mu & ( \{\, y \in [0,1] : \mu(B(y,r_m))\tau_{B(y,r_m)}(x)
        \geq t \,\}) \\
      & \geq \mu \Bigl( \Bigl \{\, y \in [0,1]\setminus
        \tilde{B} : \mu(B(y,r_{\ell_{m}(y)}(y,t(1+2\dd))))
        \tau_{B(y,r_{\ell_{m}(y)}(y,t(1+2\dd)))}(x) \\
      & \hspace{9cm} \geq (1+ 2\dd)t \, \Bigr\} \Bigr)  \\
      & \geq \mu \Bigl( \Bigl\{ \, y \in
        [0,1] : \mu(B(y,r_{\ell_{m}(y)}(y,t(1+2\dd))))
        \tau_{B(y,r_{\ell_{m}(y)}(y,t(1+2\dd)))}(x) \\
      & \hspace{8cm} \geq (1+ 2\dd)t \, \Bigr\} \Bigr) -
        \mu(\tilde{B}) \\
      & \geq \ee^{-t \cdot (1+2\dd)} - \frac{\e}{8}-\frac{\e}{8}
        > \ee^{-t} - \frac{3\e}{8},
\end{align*}
where we again used equation \eqref{eq:cont} in the last
estimate.

Altogether we showed
\[
  \abs{ \mu (\{\, y : \mu(B(y,r_m))\tau_{B(y,r_m)}(x) \geq t\,\}
    )-\ee^{-t}}< \frac{3\e}{4},
\]
which finishes the proof.

\section{Proof of Theorem~\ref{the:hts}}

Similarly to the proof of Theorem~\ref{the:vts} we put
\[
  F_n (y) := \mu \{\, x : \tau_{B(y,r_n)} (x) > t / \mu (B(y,r_n))
  \,\},
\]
where $r_n$ is chosen such that $\mu (B(y,r_n)) = t/n$. We have
\begin{align*}
 F_n (y) & = \int \mathbbm{1}_{A_n} (x,y) \, \mathrm{d} \mu (x) =
  1 + \sum_{m=1}^n (-1)^m \int S_{m,n} (x,y) \, \mathrm{d}\mu (x) \\
  &= 1 + \sum_{m=1}^n (-1)^m \int \tilde{S}_{m,n} (x,y) \, \mathrm{d}\mu (x),
\end{align*}
where we used assumption~\eqref{eq:NoReturnCond} in the last step.
We define $M_n^\mathrm{odd}$ and $M_n^\mathrm{even}$ as in the
proof of Theorem~\ref{the:vts} and define $G_n^\mathrm{odd}$ and
$G_n^\mathrm{even}$ by
\[
  G_n^\mathrm{odd} (y) =
  1 + \sum_{m=1}^{M_n^\mathrm{odd}} (-1)^m\int \tilde{S}_{m,n} (x,y) \, \mathrm{d}\mu (x),
\]
and similarly for $G_n^\mathrm{even}$.
By the inclusion--exclusion principle, we have
\[
  G_n^\mathrm{odd} \leq F_n \leq G_n^\mathrm{even}.
\]
The goal is to prove that $\lim_{n \to \infty} G_n^\mathrm{odd}
(y) = \lim_{n \to \infty} G_n^\mathrm{odd} (y) = e^{-t}$ and from
this we conclude that $\lim_{n \to \infty} F_n (y) = e^{-t}$. The
analysis of $G_n^\mathrm{odd}$ and $G_n^\mathrm{even}$ is the
same, and from now on we only write $G_n$ to denote whichever one
of them.

As before we choose $c_2 > 4 \tau^{-1}>\tau^{-1} >c_1$. For $k \in \hat{K}_{m,n,p}$, we now estimate
\[
  P_k = \int \mathbbm{1}_{B_y} (T^{k_1} x) \ldots
  \mathbbm{1}_{B_y} (T^{k_m} x) \, \mathrm{d} \mu (x).
\]
If $k\in\hat{K}_{m,n,0}$, then exponential mixing for intervals against $L^1$ leads to
\[
  \biggl(\frac{t}{n} - \frac{c}{n^{\tau c_2}} \biggr)^m
  \leq P_k \leq
  \biggl(\frac{t}{n} + \frac{c}{n^{\tau c_2}} \biggr)^m
\]
If $k\in \hat{K}_{m,n,p}$ for $p>1$, then we have that
\[
0\leq   P_k \leq
  \biggl(\frac{t}{n} + \frac{c}{n^{\tau c_2}} \biggr)^{m-p}
  \biggl(\frac{t}{n} + \frac{c}{n^{\tau c_1}} \biggr)^{p}.
\]
By Lemma~\ref{lem:firstcounting}, we have that the number of such
sequences is at most
\[
  \# \hat{K}_{m,n,p} \leq c^p \frac{(\log n)^{p (1 +
      2\alpha)}}{n^p} \# \tilde{K}_{m,n}.
\]
The number of sequences with $p=0$ is at least
\[
  \# \hat{K}_{m,n,0} \geq \# \tilde{K}_{m,n} - \sum_{p=1}^\infty
  c^p \frac{(\log n)^{p (1 + 2\alpha)}}{n^p} \# \tilde{K}_{m,n}
  \geq \biggl(1 - \frac{c}{n^{1-\varepsilon}} \biggr) \#
  \tilde{K}_{m,n}
\]
for any $\varepsilon >0$ if $n$ is sufficiently large. In the following we take $0<\varepsilon<\tau c_1$. 

These estimates now imply that
\begin{equation*}
  \sum_k P_k \geq \biggl(1 - \frac{c}{n^{1-\varepsilon}} \biggr)
  \biggl(\frac{t}{n} - \frac{c}{n^{\tau c_2}} \biggr)^m \#
  \tilde{K}_{m,n}.
\end{equation*}
Since $m \leq M_n$ we have by the Bernoulli inequality that
\begin{align*}
  \biggl( \frac{t}{n} - \frac{c}{n^{\tau c_2}} \biggr)^m
  & = \biggl(
    \frac{t}{n} \biggr)^m \biggl(1 - \frac{c/t}{n^{\tau c_2 - 1}}
    \biggr)^m \geq \biggl( \frac{t}{n} \biggr)^m \biggl(1 - \frac{c
    m}{n^{\tau c_2 - 1}} \biggr) \\
  & \geq \biggl( \frac{t}{n} \biggr)^m
    \biggl(1 - \frac{c}{n^{\tau c_2 - 1 - \varepsilon}} \biggr).
\end{align*}
Since $\tau c_2 >4$ we have
$1 - \varepsilon < \tau c_2 - 1 - \varepsilon$ and therefore 
\begin{equation}
  \label{eq:Pklowerbound}
  \sum_k P_k \geq \biggl(1 - \frac{c}{n^{1-\varepsilon}} \biggr)
  \biggl( \frac{t}{n} \biggr)^m \# \tilde{K}_{m,n}.
\end{equation}
As in \eqref{eq:expectationestimate}, we use similar estimates to conclude for large $n$, again using the Bernoulli inequality, that
\begin{align}
\nonumber
\sum_k P_k  & \leq \biggl(\frac{t}{n} + \frac{c}{n^{\tau c_2}} \biggr)^m  \biggl(1+c\frac{(\log n)^{1+2\alpha}}{n^{\tau c_1}}\biggr)   \# \tilde{K}_{m,n} \\
& \leq \biggl(1+\frac{c}{n^{\tau c_1-\varepsilon}}\biggr) \biggl( \frac{t}{n} \biggr)^m  \# \tilde{K}_{m,n}.\label{eq:Pkupperbound}
\end{align}

From the
estimates on $\sum_k P_k$ we then get that
\[
  \biggl| G_n (y) - \biggl[ 1 + \sum_{m=1}^{M_n} (-1)^m \biggl(
  \frac{t}{n} \biggr)^m \# \tilde{K}_{m,n} \biggr] \biggr| \leq
  \sum_{m = 1}^{M_n} \frac{c}{n^{\tau c_1-\varepsilon}} \biggl(
  \frac{t}{n} \biggr)^m \# \tilde{K}_{m,n} \leq
  \frac{c}{n^{\tau c_1-\varepsilon}} e^t.
\]
Since, by Lemma~\ref{lem:DeviationZero}, we have that
\[
  1 + \sum_{m=1}^{M_n} (-1)^m \biggl( \frac{t}{n} \biggr)^m \#
  \tilde{K}_{m,n} \to e^{-t}
\]
as $n \to \infty$, this proves that $G_n (y) \to e^{-t}$ as
$n \to \infty$.

An argument similar to the one in Section~\ref{sec:vtsproof_extension}, but simpler, allows us to extend this to all sequences $(r_n)$ with $r_n \searrow 0$ meaning that we have the convergence for $r\to 0$.

\section{Proof of Theorem~\ref{the:poisson}}

For $r>0$ we let
\[
  \phi_n (z) = \int \prod_{k = 1}^n ( 1 - z \mathbbm{1}_{B (y,r)}
    (T^k x)  ) \, \mathrm{d} \mu (x).
\]
Clearly, $\phi_n$ is a polynomial of degree at most $n$. We put
$\psi_n (z) = \phi_n (1 - z)$ and write
\[
  \psi_n (z) = a_{n,0} + a_{n,1} z + \ldots + a_{n,n} z^n.
\]

\begin{proposition}\label{prop:Coeff}
  We have that
  \[
    a_{n,p} = \mu \{\, x : p = \# \{\, j \leq n : T^j (x) \in
    B(y,r) \,\} \,\}.
  \]
\end{proposition}

\begin{proof}
  We may write
  \begin{align*}
    \psi_n (z)
    &= \int \prod_{k = 1}^n ( 1 - (1 - z) \mathbbm{1}_{B (y,r)}
      (T^k x)  ) \, \mathrm{d} \mu (x) \\
    &= \int \prod_{k = 1}^n ( \mathbbm{1}_{B(y,r)} (T^k x) +
      \mathbbm{1}_{\complement B(y,r)} (T^k x) - (1 - z)
      \mathbbm{1}_{B (y,r)} (T^k x)  ) \, \mathrm{d} \mu
      (x) \\
    &= \int \prod_{k = 1}^n ( \mathbbm{1}_{\complement
      B(y,r)} (T^k x) + z \mathbbm{1}_{B (y,r)} (T^k x)  )
      \, \mathrm{d} \mu (x),
  \end{align*}
  from which the statement follows by expanding the product and collecting the coefficients for each power of $z$.
\end{proof}

As before we fix $t>0$ and take $r=r_n(y)$ such that $\mu(B(y,r))=t/n$. The strategy now is to prove that $\phi_n (z) \to e^{-tz}$ in a
strong enough sense that we can conclude, not only that $\psi_n
(z) \to e^{-t(1-z)}$, but that for each fixed $p$,
\[
  \lim_{n \to \infty} a_{n,p} = c_p,
\]
where $c_p = \frac{t^p e^{-t}}{p!}$. Note that these are the
coefficients in the power series for $e^{-t(1-z)}=e^{-t}\sum^{\infty}_{p=0}\frac{(tz)^p}{p!}$.

We first write
\[
  \phi_n (z) = b_{n,0} + b_{n,1} z + \ldots + b_{n,n} z^n,
\]
and prove that
\[
  \lim_{n \to \infty} b_{n,p} = \frac{(-t)^p}{p!}.
\]
Recall that by the power series definition of $e^z$, 
\[
  e^{-tz} = \sum_{p=0}^\infty d_p z^p, \quad d_p =
  \frac{(-t)^p}{p!}.
\]

To analyze the convergence of $\psi_n$, we need to connect the
coefficients of $\psi_n$ to those of $\phi_n$. We have that
\begin{align*}
  \psi_n (z)
  &= a_{n,0} + a_{n,1} z + \ldots + a_{n,n} z^n \\
  &= b_{n,0} + b_{n,1} (1-z) + \ldots + b_{n,n} (1-z)^n, \\
\end{align*}
so that, via the binomial formula, one obtains
\begin{equation}
  \label{eq:a-formula}
  a_{n,p} = \sum_{j = p}^n b_{n,j} (-1)^p \binom{j}{p}.
\end{equation}
Similarly, we have
\begin{equation}
  \label{eq:c-formula}
  c_p = \sum_{j = p}^\infty d_j (-1)^p \binom{j}{p}.
\end{equation}
We are going to estimate $|b_{n,j} - d_j|$ in order to get an
estimate on $|a_{n,p} - c_p|$.

As before, we let $\tilde{K}_{p,n}$ be the set of increasing sequences $k=(k_1,\dots, k_p) \in K_{p,n}$ such that $k_{j+1}-k_j > c_1 \log n$. By expanding the product in the definition of $\phi$ and applying assumption~\eqref{eq:NoReturnCond} we see that
\begin{align*}
  (-1)^p b_{n,p} & = \sum_{k \in K_{p,n}} \int
  \mathbbm{1}_{B(y,r)} (T^{k_1} x) \ldots \mathbbm{1}_{B(y,r)} (T^{k_p}
  x) \, \mathrm{d} \mu(x) \\
 & = \sum_{k \in \tilde{K}_{p,n}} \int
  \mathbbm{1}_{B(y,r)} (T^{k_1} x) \ldots \mathbbm{1}_{B(y,r)} (T^{k_p}
  x) \, \mathrm{d} \mu(x)
\end{align*}
We again choose $c_2>4\tau^{-1}>\tau^{-1}>c_1$. By \eqref{eq:Pkupperbound} and \eqref{eq:Pklowerbound} combined
with Lemma~\ref{lem:binomial}, we have
\[
  (-1)^p b_{n,p} \leq \biggl(\frac{t}{n}\biggr)^p
  \biggl(1 + \frac{c}{n^{\tau c_1-\varepsilon}} \biggr) \binom{n}{p}
  \]
  and  \[ (-1)^p b_{n,p} \geq \biggl(\frac{t}{n}\biggr)^p  \biggl( 1
  - \frac{c}{n^{1-\varepsilon}} \biggr) \binom{n}{p}.
\]

Therefore,
\begin{equation}
  \label{eq:bnpestimate}
  1 - \frac{c}{n^{1-\varepsilon}} \leq \frac{b_{n,p}}{(-t)^p
    \frac{1}{n^p} \binom{n}{p}} \leq 1 + \frac{c}{n^{\tau c_1-\varepsilon}}.
\end{equation}

For a fixed $p$, we have
\[
  \lim_{n \to \infty} \frac{1}{n^p} \binom{n}{p} = \frac{1}{p!}.
\]
In fact,
\begin{align*}
  \frac{1}{n^p} \binom{n}{p}
  & = \frac{1}{p!} \frac{n (n-1) \ldots
    (n - p + 1)}{n^p} = \frac{1}{p!} \biggl(1 - \frac{1}{n}
    \biggr) \ldots \biggl( 1 - \frac{p-1}{n} \biggr) \\
  & \geq \frac{1}{p!} \biggl(1 - \frac{p}{n}
    \biggr)^{p}.
\end{align*}
Hence, there exists a $c>0$ so that if $p < \frac{n}{2}$, then
\[
  \frac{1}{n^p} \binom{n}{p} \geq \frac{1}{p!} e^{p
    \log (1 - p/n)} \geq \frac{1}{p!} e^{- \frac{c p^2}{n}} \geq \frac{1 -
      \frac{c p^2}{n}}{p!}
\]
so that
\begin{equation}
  \label{eq:bnpestimate2}
  \biggl( 1 - \frac{c}{n^\frac{1}{2}} \biggr) \frac{1}{p!} \leq
  \frac{1}{n^p} \binom{n}{p} \leq \frac{1}{p!},
\end{equation}
provided $p \leq n^\frac{1}{4}$ and $n$ is large.  Combining
\eqref{eq:bnpestimate} and \eqref{eq:bnpestimate2}, we get that
\begin{equation}
  \label{eq:b-d-estimate}
  1 - \frac{c}{n^{\frac{1}{2}}} \leq \frac{b_{n,p}}{d_p}
  \leq 1 + \frac{c}{n^{\tau c_1-\varepsilon}}, \quad \text{ for } p \leq n^\frac{1}{4}.
\end{equation}

We can now use our obtained estimates to estimate
$|a_{n,p} - c_p|$. Using \eqref{eq:a-formula} and
\eqref{eq:c-formula}, we first write
\begin{align*}
  |a_{n,p} - c_p|
  & \leq \sum_{j = p}^{n} |b_{n,j} - d_j|
    \binom{j}{p} + \sum_{j = n+1}^\infty |d_j| \binom{j}{p} \\
  & \leq \sum_{j = p}^{n^\frac{1}{4}} |b_{n,j} - d_j|
    \binom{j}{p} + \sum_{j = n^\frac{1}{4} + 1}^n (|b_{n,j}| +
    |d_j|) \binom{j}{p} + \sum_{j = n+1}^\infty |d_j|
    \binom{j}{p} \\
  & = \sum_{j = p}^{n^\frac{1}{4}} |b_{n,j} - d_j|
    \binom{j}{p} + \sum_{j = n^\frac{1}{4} + 1}^n |b_{n,j}|
    \binom{j}{p} + \sum_{j = n^\frac{1}{4} + 1}^\infty |d_j|
    \binom{j}{p}.
\end{align*}
Let us estimate the three sums separately. From
\eqref{eq:b-d-estimate} we have that
\begin{align*}
  \sum_{j = p}^{n^\frac{1}{4}} |b_{n,j} - d_j| \binom{j}{p}
  & \leq \frac{c}{n^{\min\left(0.5,\tau c_1-\varepsilon\right)}} \sum_{j = p}^{n^\frac{1}{4}}
    \frac{t^j}{j!}  \binom{j}{p} \leq \frac{c}{n^{\min\left(0.5,\tau c_1-\varepsilon\right)}}
    \sum_{j = p}^{\infty} \frac{t^j}{p! (j-p)!} \\
    & \leq \frac{ct^p}{p!n^{\min\left(0.5,\tau c_1-\varepsilon\right)}}
    \sum_{j = p}^{\infty} \frac{t^{j-p}}{(j-p)!} \leq \frac{c
    t^p e^t}{p! n^{\min\left(0.5,\tau c_1-\varepsilon\right)}}.
\end{align*}
By \eqref{eq:bnpestimate}, we have $|b_{n,j}| \leq (1+c) t^j n^{-j}
\binom{n}{j}$. Therefore,
\begin{align*}
  \sum_{j = n^\frac{1}{4} + 1}^n |b_{n,j}| \binom{j}{p} & \leq
  \sum_{j = n^\frac{1}{4} + 1}^\infty (1+c) t^j n^{-j}
  \binom{n}{j} \binom{j}{p} \leq 
  \sum_{j = n^\frac{1}{4} + 1}^\infty (1+c) t^j n^{-j}
  \frac{n^j}{p!(j-p)!} \\
  & \leq \frac{c t^p}{p!} \sum_{j =
    n^\frac{1}{4} + 1}^\infty \frac{t^{j-p}}{(j-p)!}.
\end{align*}
The last sum is the remainder in a Taylor expansion of the
exponential function,
\[
  e^t = \sum_{j = p}^{n^\frac{1}{4}} \frac{t^{j-p}}{(j-p)!} + \sum_{j =
    n^\frac{1}{4} + 1}^\infty \frac{t^{j-p}}{(j-p)!}.
\]
We can now use the Lagrange form of the remainder and Stirling's formula to generously estimate
\[
\sum_{j = n^\frac{1}{4} + 1}^n |b_{n,j}| \binom{j}{p} \leq \frac{c
    t^p e^t}{p! n}.
\]
Finally, we have by the same reason that
\[
  \sum_{j = n^\frac{1}{4} + 1}^\infty |d_j| \binom{j}{p} =
  \frac{t^p}{p!} \sum_{j = n^\frac{1}{4} + 1}^\infty
  \frac{t^{j-p}}{(j-p)!} \leq \frac{c t^p e^t}{p! n}.
\]
Hence, there is a constant $C_p$ such that
\[
  |a_{n,p} - c_p| \leq C_p n^{-\min\left(0.5,\tau c_1-\varepsilon\right)}.
\]

By Proposition~\ref{prop:Coeff} this shows
\[
\lim_{n \to \infty} \mu \{\, x : p = \# \{\, j \leq n : T^j (x) \in
    B(y,r_n(y,t)) \,\} \,\} = \frac{t^p e^{-t}}{p!},
\]
where $r_n(y,t)$ is chosen such that
$\mu\left(B(y,r_n(y,t))\right)=t/n$. An argument similar to the
one in Section~\ref{sec:vtsproof_extension}, but simpler, allows
us to extend this to all sequences $(r_n)$ with $r_n \searrow 0$
meaning that we have the convergence for $r\to 0$.

\section{Appendix on mixing in the quadratic family}
\label{sec:appendix}

Throughout this section, we study the quadratic family
$T_a (x) = ax(1-x)$. For a set of parameters $a$ of positive Lebesgue measure there exists a $T_a$-invariant measure $\mu_a$ which is
absolutely continuous with respect to Lebesgue measure.
We let $\rho_a$ denote the density of $\mu_a$, when it exists,
and we let $m$ be the Lebesgue measure on $\mathbbm{R}$.

Young \cite{Young92} has stated exponential mixing in the
quadratic family for functions of bounded variation, but the
proof \cite[Line~4, Page~134]{Young92} actually gives
\[
  \biggl| \int \phi \cdot \psi \circ T^n_a \, \mathrm{d}\mu_a -
  \int \phi \, \mathrm{d} \mu_a \int \psi \, \mathrm{d} \mu_a
  \biggr| \leq C \lVert \phi \rVert_{BV} \lVert \psi
  \rVert_{\infty} \tau^n.
\]

The purpose of this section is to show exponential mixing for
some Sobolev space against $L^1$.  Let
$a_t (\xi) = (1 + \xi^2)^{t/2}$. We will consider fractional
Sobolev spaces $H_p^t$ of exponent $t > 0$ and integrability
$p > 1$, also called Bessel potential spaces, defined by the norm
\[
  \lVert f \rVert_{H_p^t} = \lVert F^{-1} (a_t
  F (f)) \rVert_{L^p (m)},
\]
where $F$ is the Fourier transform. Alternatively, $f$ belongs to
$H_p^t$ if $f = G_t * g$, where $G_t$ is the Bessel potential
defined by $F(G_t) = a_t$, and $g \in L^p (m)$. We then have
$\lVert f \rVert_{H_p^t} = \lVert g \rVert_{L^p (m)}$.

\begin{lemma}
  There are numbers $s$ and $q$ with $0< s < 1/q < 1$, and a
  Cantor set $\Delta$ of positive Lebesgue measure, such that for
  $a \in \Delta$, there are constants $\tilde{C}_0, \tau > 0$
  such that
  \begin{equation}
    \label{eq:sobolevmixing2}
    \biggl| \int f \circ T^n_a g \, \mathrm{d} \mu_a - \int f \,
    \mathrm{d} \mu_a \int g \, \mathrm{d} \mu_a \biggr| \leq
    \tilde{C}_0 e^{-\tau n} \lVert f \rVert_{L^1(\mu_a)} \lVert g
    \rVert_{H_q^s},
  \end{equation}
  holds for all $n$ and all functions $f$ and $g$ with
  $\lVert f \rVert_{L^1(\mu_a)}, \lVert g \rVert_{H_q^s} < \infty$. In
  particular, there is a constant $\tilde{C}$ such that
  \begin{equation}
    \label{eq:intervalmixing}
    \biggl| \int f \circ T^n_a \mathbbm{1}_I \, \mathrm{d} \mu_a -
    \mu(I) \int f \, \mathrm{d} \mu_a \biggr| \leq \tilde{C}
    e^{-\tau n} \lVert f \rVert_{L^1(\mu_a)},
  \end{equation}
  holds for all $n$, all functions $f$ with
  $\lVert f \rVert_{L^1(\mu_a)} < \infty$ and all intervals
  $I \subset [0,1]$.
\end{lemma}

\begin{proof}
  We shall use the following mixing result (see for instance
  Aspen\-berg--Baladi--Persson
  \cite[Proposition~2.5]{AspenbergBaladiPersson}) for the
  quadratic family, which is valid for $a\in\Delta$, where
  $\Delta$ is a Cantor set of parameters of positive measure. Whenever
  $t > 0$ and $p > 1$, there is $c_0 > 0$ and $\tau>0$
  such that
  \begin{equation}
    \label{eq:sobolevmixing}
    \biggl| \int \phi \cdot \psi \circ T^n_a \, \mathrm{d} m - \int
    \phi \, \mathrm{d} m \int \psi \, \mathrm{d} \mu_a \biggr|
    \leq c_0 \lVert \phi \rVert_{H_p^t} \lVert \psi \rVert_{L^1
      (\mu_a)} e^{-\tau n}.
  \end{equation}
  for all $n$ and all functions $\phi$ and $\psi$ with $\lVert \psi \rVert_{L^1(\mu_a)},\lVert \phi\rVert_{H_p^t}<\infty$.
  In order for \eqref{eq:sobolevmixing} to imply \eqref{eq:sobolevmixing2} we need to argue that the appearances of the measure $m$ in \eqref{eq:sobolevmixing} may be replaced by $\mu_a$ at the cost of adjusting the parameters of the fractional Sobolev space $H_p^t$. This argument goes as follows.

  We will prove that there is
  $\tilde{C}_0 > 0$ and $0 < s < 1/q < 1$ such that
  \eqref{eq:sobolevmixing2} holds.
    To prove this, we shall use that, if $0 \leq u < 1/2$
  and $1 < r < \frac{2}{1 + 2s}$, then the density $\rho_a$ of $\mu_a$
  belongs to $H_{r}^{u}$
  \cite[Sublemma~2.7]{AspenbergBaladiPersson}. We fix such $u$
  and $r$. 

  We now choose $s, q$ and $t, p$. We require that
  \begin{align*}
    u, s \geq t, &&
    u - t \geq \frac{1}{r} - \frac{1}{p}, \\
    s - t \geq \frac{1}{q} - \frac{1}{p}, &&
    u + s - t \geq \frac{1}{r} + \frac{1}{q} -
    \frac{1}{p}. 
  \end{align*}

  Now, take $\phi \in H_{q}^{s}$. From Bezhadan and Holst
  \cite[Theorem~5.1]{BezhadanHolst}, we get that
  $\phi \rho_a \in H_p^t$ and that
  $\lVert \phi \rho_a \rVert_{H_p^t} \leq c_1 \lVert \phi
  \rVert_{H_{q}^{s}} \lVert \rho_a \rVert_{H_{r}^{u}}$. The
  mixing property \eqref{eq:sobolevmixing} now implies that
  \begin{align*}
    \biggl| \int \phi \cdot \psi \circ T^n_a \, \mathrm{d} \mu_a
    &- \int \phi \, \mathrm{d} \mu_a \int \psi \, \mathrm{d}
      \mu_a \biggr|\\
    & = \biggl| \int \phi \rho_a \cdot \psi \circ T_a \,
      \mathrm{d} m - \int \phi \rho_a \, \mathrm{d} m \int \psi
      \, \mathrm{d} \mu_a \biggr| \\
    & \leq c_0 \lVert \phi \rho_a \rVert_{H_p^t} \lVert \psi
      \rVert_{L^1(\mu_a)} e^{-\tau n}
      \\
      &\leq c_0 c_1 \lVert \phi \rVert_{H_{q}^{s}} \lVert \rho_a
      \rVert_{H_{r}^{u}} \lVert \psi \rVert_{L^1(\mu_a)} e^{-\tau n}\\
    & = \tilde{C}_0 \lVert \phi \rVert_{H_{q}^{s}} \lVert \psi
      \rVert_{L^1(\mu_a)}e^{-\tau n}, \quad \tilde{C}_0 = c_0 c_1 \lVert \rho_a
      \rVert_{H_{r}^{u}} < \infty,
  \end{align*}
  where the use of mixing requires that $p > 1$.

  Hence, we have proved \eqref{eq:sobolevmixing2}, provided we
  can at the same time satisfy the above conditions and
  \[
    0 < s < \frac{1}{q} < 1
  \]
  (the latter to ensure that $H_{q}^{s}$ contains the indicator
  functions of intervals, as will be explained below). Putting
  all requirements together we need
  \begin{align*}
    0 < u < \frac{1}{2}, &&
    1 < r < \frac{2}{1 + 2 u}, &&
    u, s \geq t, \\
    u - t \geq \frac{1}{r} - \frac{1}{p}, &&
    s - t \geq \frac{1}{q} - \frac{1}{p}, \\
    u + s - t \geq \frac{1}{r} + \frac{1}{q} -
    \frac{1}{p}, &&
    0 < s < \frac{1}{q} < 1.
  \end{align*}
  Taking $t = u$, and $q = p$, this simplifies to
  \begin{align*}
    0 < u < \frac{1}{2}, &&
    1 < r < \frac{2}{1 + 2 u}, \\
    s \geq u, &&
    r \geq q, \\
    s \geq \frac{1}{r}, &&
    0 < s < \frac{1}{q} < 1.
  \end{align*}
  It is possible to choose such numbers. Let for instance
  \begin{align*}
    u = 0.1, && s &= 0.65, && t = 0.1\\
    r = 1.6, && q &= 1.5, && p = 1.5.
  \end{align*}
  With such a choice, \eqref{eq:sobolevmixing2} is proved.

  To prove \eqref{eq:intervalmixing}, it is now enough to notice
  that since $0 < s < 1/q < 1$, there is a constant $c_2$ such
  that $\lVert \mathbbm{1}_I \rVert_{H_q^s} \leq c_2$ for all
  intervals $I \subset [0,1]$. But this is a direct consequence
  of Thomine \cite[Lemma~4.2]{Thomine}, and
  \eqref{eq:intervalmixing} is proved with $\tilde{C} =
  \tilde{C}_0 c_2$.
\end{proof}


\begin{thebibliography}{11}

\bibitem{AspenbergBaladiPersson} M. Aspenberg, V. Baladi,
  T. Persson, \emph{A parameter ASIP for the quadratic family},
  Ergodic Theory and Dynamical Systems 45 (2025), 663--703.
  
\bibitem{BenedicksCarleson} M. Benedicks, L. Carleson, \emph{On
    iterations of $1-ax^2$ on $(-1,1)$}, Ann. of Math. (2) 122
  (1985), no. 1, 1--25.
  
\bibitem{BezhadanHolst} A. Bezhadan, M. Holst,
  \emph{Multiplication in Sobolev space, revisited}, Arkiv för
  matematik 59 (2021), 275--306.

\bibitem{Bruinetal} H. Bruin, B. Saussol, S. Troubetzkoy, S
  Vaienti, \emph{Return time statistics via inducing}, Ergodic
  Theory Dynam. Systems 23 (2003), no. 4, 991--1013.
  
\bibitem{ChazottesCollet} J.-R. Chazottes, P. Collet,
  \emph{Poisson approximation for the number of visits to balls
    in non-uniformly hyperbolic dynamical systems}, Ergodic
  Theory Dynam. Systems 33 (2013), no. 1, 49--80.

\bibitem{Collet} P. Collet, \emph{Some ergodic properties of maps
    of the interval}, Dynamical systems (Temuco, 1991/1992),
  55--91, Travaux en Cours, 52, Hermann, Paris, 1996.
  
\bibitem{Denkeretal} M. Denker, M. Gordin, A. Sharova, \emph{A
    Poisson limit theorem for toral automorphisms}, Illinois
  J. Math. 48 (2004), no. 1, 1--20.

\bibitem{FreitasFreitasMagalhaes} A. Freitas, J. Freitas,
  M. Magalhães, \emph{Complete convergence and records for
    dynamically generated stochastic processes},
  Trans. Amer. Math. Soc. 373 (2020), no. 1, 435--478.
  
\bibitem{FreitasFreitasTodd} A. Freitas, J. Freitas, M. Todd,
  \emph{Hitting time statistics and extreme value theory},
  Probab. Theory Related Fields 147 (2010), no. 3--4, 675--710.

\bibitem{Haydn} N. T. A. Haydn, \emph{Entry and return times
    distribution}, Dyn. Syst. 28 (2013), no. 3, 333--353.

\bibitem{Hirata} M. Hirata, \emph{Poisson law for Axiom A
    diffeomorphisms}, Ergodic Theory Dynam. Systems 13 (1993),
  533--556.
  
\bibitem{Hirataetal} M. Hirata, B. Saussol, S. Vaienti,
  \emph{Statistics of return times: a general framework and new
    applications}, Comm. Math. Phys. 206 (1999), no. 1, 33--55.
    
\bibitem{HofbauerKeller} F. Hofbauer, G. Keller, \emph{Ergodic
    properties of invariant measures for piecewise monotonic
    transformations}, Math. Z. 180 (1982), no. 1, 119--140.

\bibitem{Hollandetal} M. Holland, M. Kirsebom, Ph. Kunde,
  T. Persson, \emph{Dichotomy results for eventually always
    hitting time statistics and almost sure growth of extremes},
  Trans. Amer. Math. Soc. 377 (2024), no. 6, 3927--3982.

\bibitem{HollandTodd} M. Holland, M. Todd, \emph{On
    distributional limit laws for recurrence}, Nonlinearity 38
  (2025), no. 7, Paper No. 075028.
  
\bibitem{Kac} M. Kac, \emph{On the notion of recurrence in
    discrete stochastic processes}, Bull. Amer. Math. Soc. 53
  (1947), 1002--1010.
  
\bibitem{Pitskel} B. Pitskel', \emph{Poisson limit law for Markov
    chains}, Ergodic Theory Dynam. Systems 11 (1991), no. 3,
  501--513.
  
\bibitem{Rychlik} M. Rychlik, \emph{Bounded variation and
    invariant measures}, Studia Mathematica 76 (1983), 69--80.
   
\bibitem{Saussol} B. Saussol, \emph{Absolutely continuous invariant
    measures for multidimensional expanding maps}, Israel Journal
  of Mathematics, 116 (2000), 223--248.

\bibitem{Thomine} D. Thomine, \emph{A spectral gap for transfer operators
    of piecewise expanding maps}, Discrete Contin. Dyn. Syst. 30
  (2011), no. 3, 917--944.
  
\bibitem{Young92} L. S.\ Young, \emph{Decay of Correlations for
    Certain Quadratic Maps}, Communications in Mathematical
  Physics 146 (1992), 123--138.

\end{thebibliography}
\end{document}